\documentclass[a4paper,12pt]{amsart}\usepackage[utf8]{inputenc}
\usepackage[english]{babel}
\usepackage{vmargin}

\usepackage{amsmath,amsthm,amsfonts,amssymb,latexsym, mathrsfs, yfonts,dsfont}
\usepackage{ stmaryrd }
\usepackage{multicol}
\usepackage{tikz-cd}
\usepackage{soul}

\setpapersize{A4}
\setmargins{2cm}      
{1.5cm}                        
{16.5cm}                      
{23.42cm}                    
{10pt}                           
{1.5cm}                           
{0pt}                             
{2cm}                           

\makeatletter
\def\newrefformat#1#2{%
  \@namedef{pr@#1}##1{#2}}
\def\prettyref#1{\@prettyref#1:}
\def\@prettyref#1:#2:{%
  \expandafter\ifx\csname pr@#1\endcsname\relax%
    \PackageWarning{prettyref}{Reference format #1\space undefined}%
    \ref{#1:#2}%
  \else%
    \csname pr@#1\endcsname{#1:#2}%
  \fi%
}
\makeatother

\newrefformat{eq}{\textup{(\ref{#1})}}
\newrefformat{cha}{Chapter \ref{#1}}
\newrefformat{sec}{Section \ref{#1}}
\newrefformat{tab}{Table \ref{#1} on page \pageref{#1}}
\newrefformat{fig}{Figure \ref{#1} on page \pageref{#1}}

\makeatletter
\def\indsym#1#2{%
  \setbox0=\hbox{$\m@th#1x$}%
  \kern\wd0%
  \hbox to 0pt{\hss$\m@th#1\mid$\hbox to 0pt{$\m@th#1^{#2}$}\hss}%
  \lower.9\ht0\hbox to 0pt{\hss$\m@th#1\smile$\hss}%
  \kern\wd0} 
\def\nindsym#1#2{%
  \setbox0=\hbox{$\m@th#1x$}%
  \kern\wd0%
  \hbox to 0pt{\mathchardef\nn="3236\hss$\m@th#1\nn$\kern1.4\wd0\hss}
  \hbox to 0pt{\hss$\m@th#1\mid$\hbox to 0pt{$\m@th#1^{#2}$}\hss}%
  \lower.9\ht0\hbox to 0pt{\hss$\m@th#1\smile$\hss}%
  \kern\wd0}

\makeatother

\makeatletter
\theoremstyle{plain}
\newtheorem{thm}{Theorem}[section]
\numberwithin{equation}{section} 
\numberwithin{figure}{section} 
\theoremstyle{plain}
\newtheorem{cor}[thm]{Corollary} 
\theoremstyle{definition}
\newtheorem{defn}[thm]{Definition}
\newtheorem{remark}[thm]{Remark}
\theoremstyle{plain}
\newtheorem{lem}[thm]{Lemma} 
\theoremstyle{plain}
\newtheorem{prop}[thm]{Proposition} 
\theoremstyle{plain}
\newtheorem{fact}[thm]{Fact}
\theoremstyle{plain}
\newtheorem{notation}[thm]{Notation}
\theoremstyle{plain}

\newtheorem{ej}[thm]{Example}

\newtheorem{preg}[thm]{Question}
\newtheorem{obs}[thm]{Observation}
\newtheorem{teo}[thm]{Theorem}

\theoremstyle{plain}

\newtheorem{afirm}[thm]{\textit{Claim}}

\def\upharpoonrightleft{\! \!\upharpoonright}

\def\A{{\mathbb A}}
\def\B{{\mathbb B}}

\def\N{{\mathbb N}}
\def\Z{{\mathbb Z}}

\def\Q{{\mathbb Q}}
\def\R{{\mathbb R}}

\def\modA{{\mathcal A}}
\def\modB{{\mathcal B}}

\def\modH{{\mathcal H}}

\def\modK{{\mathcal K}}
\def\modL{{\mathcal L}}
\def\modM{{\mathcal M}}
\def\modN{{\mathcal N}}

\def\modU{{\mathcal U}}

\def\sbestrucel{{\preccurlyeq}}

\def\forkindep{\mathrel{\raise0.2ex\hbox{\ooalign{\hidewidth$\vert$\hidewidth\cr\raise-0.9ex\hbox{$\smile$}}}}}
\def\notind#1#2{#1\setbox0=\hbox{$#1x$}\kern\wd0
\hbox to 0pt{\mathchardef\nn=12854\hss$#1\nn$\kern1.4\wd0\hss}
\hbox to 0pt{\hss$#1\mid$\hss}\lower.9\ht0 \hbox to 0pt{\hss$#1\smile$\hss}\kern\wd0}

\title{The SB-property on metric structures.}

\author{Camilo Argoty}
\address{Universidad Militar Nueva Granada, Cra 11 No 101-80, Bogot\'{a}, Colombia}

\author{Alexander Berenstein}
\address{Universidad de los Andes, Cra 1 No 18A-12, Bogot\'{a}, Colombia}

\author{Nicol\'as Cuervo Ovalle}
\address{Universidad de los Andes, Cra 1 No 18A-12, Bogot\'{a}, Colombia}
\date{}

\thanks{\hspace{-0.4cm}2020 Mathematics Subject Classification. 03C45, 03C66.
\\ The authors would like to thank the referees for valuable feedback. The authors would also like to thank the support of the project CIAS-3136 Universidad Militar Nueva Granada.}

\begin{document}

\vspace{-0.7cm}\begin{abstract}
    A complete theory $T$ has the \textit{Schröder-Bernstein property} or simply the \textit{SB-property} if any pair of elementarily bi-embeddable models are isomorphic. This property has been studied in the discrete first-order setting and can be seen as a first step towards classification theory. This paper deals with the SB-property on continuous theories.  Examples of complete continuous theories that have this property include Hilbert spaces and any completion of the theory of probability algebras. We also study a weaker notion, the SB-property up to perturbations. This property holds if any two elementarily bi-embeddable models are isomorphic up to perturbations. We prove that the theory of Hilbert spaces expanded with a bounded self-adjoint operator has the SB-property up to perturbations of the operator and that the theory of atomless probability algebras with a generic automorphism have the SB-property up to perturbations of the automorphism. We also study how the SB-property behaves with respect to randomizations. Finally we prove, in the continuous setting, that if $T$ is a strictly stable theory then $T$ does not have the SB-property.
\end{abstract}

\maketitle

\section{Introduction}

\begin{defn}
We say that a complete theory $T$ has the \textit{Schröder-Bernstein property}, or simply, the \textit{SB-property}, if for any two models $\mathcal{M},\mathcal{N}\models T$ that are elementarily bi-embeddable, i.e, there exists elementary embeddings $\varphi \colon \mathcal{M}\to\mathcal{N}$ and $\psi \colon \mathcal{N}\to\mathcal{M}$, we have that $\mathcal{M}\cong\mathcal{N}$.
\end{defn}
One motivation for studying the SB-property is that if $T$ has this property, then $T$ would be a theory for which we have a ``good understanding'' of its models, in terms that they are classified by some reasonable collection of invariants. For example, by Morley's theorem, see \cite{morley1967countable}, if $T$ is countable and $\aleph_1$-categorical then the models
of $T$ are classified by a single invariant cardinal number that is preserved by elementary embeddings, so $T$ would have SB-property; however SB-property is a weaker condition than $\aleph_1$-categoricity, for example the theory of an infinite set with a predicate which is infinite and coinfinite has the SB-property and it is not $\aleph_1$-categorical. The SB-property for first order theories have been well studied in \cite{goodrick2007does},\cite{goodrick2014schroder} and \cite{nurmagambetov1989characterization}. Among the main results we can find:
\begin{fact}[Theorem 1 in \cite{nurmagambetov1989characterization}]
If $T$ is $\omega$-stable, then T has the SB property if and only if $T$ is non-multidimensional.
\end{fact}
and  
\begin{fact}[see \cite{goodrick2007does}]\label{superstable-SB}
If $T$ is not superstable, then T does not have SB-property. Furthermore, if $T$ is unstable, then for any cardinal $\kappa$, there is an infinite collection of $\kappa$-saturated models of $T$ which are pairwise elementarily bi-embeddable but pairwise nonisomorphic.
\end{fact}

These results show that there is an interesting overlap between classification theory and SB-property. The purpose of this paper is to study the \textit{Schröder-Bernstein property} in the continuous context for three families of theories: expansions of Hilbert spaces, expansions of probability algebras and randomizations. This paper is divided as follows. 
\\
\\
First, in section $2$ we will focus on the theory of infinite dimensional Hilbert spaces, and its expansions adding a bounded self-adjoint operator. In this section we will prove that infinite dimensional Hilbert spaces and also the theory obtained by adding a projection operator as a unary function, known as \textit{beautiful pairs of Hilbert spaces}, satisfies the SB-property. We also study the weaker notion of being \textit{isomorphic up to perturbations}. We prove that the theory obtained by adding as a function a bounded self-adjoint operator satisfies the \textit{SB-property up to perturbation}. 
\\
\\
In section 3 we will concentrate on the theory of probability algebras. We will prove using Maharam's Theorem (see \cite{maharam1942homogeneous}), that the theory of atomless probability algebras satisfies the SB-property. We also extend the arguments to prove that any completion of the theory of probability algebras also satisfies the SB-property. We then study the expansion of the theory by a generic automorphism, known as $APrA$. We show the theory $APrA$ does not have the SB-property but that a weaker version holds: it has the SB-property up to perturbations of the automorphism. This last theory is stable, not superstable, and $\aleph_0$-stable up to perturbations (see \cite{BenYac-Berens-2008perturbations}).\\
\\
In section 4 we study Randomizations. Informally, a randomization of a first order (discrete) model, $\modM$, is a two sorted metric structure which consist of a sort of events and a sort of functions with values on the model, usually understood as \textit{random variables}. More generally, given a complete first order theory $T$, there is a complete continuous theory known as randomized theory, $T^R$, which is the common theory of all randomizations of models of $T$. Randomizations were introduced first by Keisler in \cite{keisler1999randomizing} and then axiomatized in the continuous setting by Ben Yaacov and Keisler in \cite{ben2009randomizations}. Since randomizations where introduced, many authors focused on examining which model theoretic properties of $T$ are preserved on $T^R$, for example, in \cite{ben2009randomizations} and \cite{yaacov2008continuous} it was shown that properties like $\omega$-categoricity, stability and dependence are preserved. Similarly in \cite{andrews2015separable}  it is proved that the existence of prime models is preserved by randomization but notions like minimal models are not preserved. Following these ideas, we prove that a first order theory $T$ with $\leq\omega$ countable models has the SB-property for countable models if and only if $T^R$ has the SB-property for separable randomizations. It remains as an open question whether the SB-property transfers to $T^R$ without these extra assumptions. 
\\
\\
In section 5 we include a proof of the first statement of Fact \ref{superstable-SB} in the continuous setting: we show that if a theory $T$ is strictly stable then $T$ does not have the SB-property. 
\\
\\
We will assume throughout the paper that the reader is familiar with the model theory
of metric structures, the background needed to understand most of the examples can be
found in \cite{ContModelTheory}. We will use some tools about perturbations in model theory, mostly basic ideas
and definitions. A reader interested in the subject can check \cite{yaacov2008perturbations}. Similarly, we will use some tools from Randomizations, while we tried to give a self-contained presentation of the notions that we need, details can be found \cite{andrews2015separable, ben2009randomizations}.

\section{Hilbert spaces and expansions with
bounded self-adjoint operators.} 

In this section we will study the SB-property in Hilbert spaces and their expansions with a bounded self-adjoint operator. We will show that both Hilbert spaces and beautiful pairs of Hilbert spaces have the SB-property. On the other hand, when we deal with expansions with a general self-adjoint operator, we could only prove an approximate version of the SB-property. We also explore connections between elementary bi-embeddability of the expansions and the Weyl-von Neumann-Berg Theorem (see Fact \ref{WeylVonNeumannBerg} and its references).

Let us recall some historical background about these expansions. Expansions of Hilbert spaces with operators were studied by Henson and Iovino. When the operator is self-adjoint, Henson proved (unpublished) that one can characterize the models of the theory of such an expansion $(\modH,A)$ in terms of the spectrum of $A$ using the Weyl-von Neumann-Berg Theorem. These ideas were extended to the setting of non-degenerate representations of $C^*$-algebras by the first author of this paper in \cite[Lemma 2.16]{argoty2013} by using a generalized version of the same theorem. Our approach to the SB-property in these expansion is again based of the spectrum of the operator and its associated spectral decomposition.

We follow standard terminology and write $IHS$ for the theory of infinite dimensional Hilbert spaces. It is a very well understood theory, it is $\aleph_0$-stable, $\aleph_0$-categorical and has quantifier elimination (for proofs see \cite[Section 15]{ContModelTheory} )
 
 \begin{prop}
 $IHS$ has the SB-property.
 \end{prop}
 \begin{proof}
 Note that if  $\mathcal{H}_0,\mathcal{H}_1\models IHS$ are such that there exists two elementary embeddings  $\varphi\colon \mathcal{H}_0\to \mathcal{H}_1$ and $\psi \colon \mathcal{H}_1\to \mathcal{H}_0$, we have that the cardinality of their orthonormal basis is the same. Since every Hilbert space is characterized by the cardinal of its orthonormal base,  then $\mathcal{H}_0\cong \mathcal{H}_1$.
 \end{proof}

Now we will consider pairs of Hilbert spaces.
 
\begin{defn}
Let $\modH\models IHS$ and let $\modL$  the language of Hilbert spaces. Let $\modL_p = \modL\cup\lbrace P\rbrace$ where $P$ is a new unary function and we consider structures of the form $(\modH, P)$, where $P \colon  \modH \rightarrow \modH$ is a projection operator, i.e, $P^2=P$ and $P^*=P$.  We denote by $IHS^P$ the theory of infinite dimensional Hilbert spaces with a linear projection. Finally let $IHS^p_\omega$ be the \textit{theory of beautiful pairs of Hilbert spaces}, which is the theory $IHS^p$ together with axioms stating that there are infinitely many pairwise orthonormal vectors $v$ satisfying $P(v) = v$ and also infinitely many pairwise orthonormal vectors $u$ satisfying $P(u) = 0$, for more details see \cite{berenstein2018hilbert}.
\end{defn}
\begin{prop}\label{SB-lovelyHilbert}
$IHS^P_\omega$ has the SB-property.
\end{prop}
\begin{proof}
Let $(\modH_1,P_1)$ and $(\modH_2,P_2)$ models of $IHS^P_\omega$ such that there exists elementary embeddings $\varphi\colon (\modH_1,P_1)\to(\modH_2,P_2)$ and $\psi\colon (\modH_2,P_2)\to(\modH_1,P_1)$. For each, $i=1,2$ let 
$$H_{P_i}=\lbrace v\in \modH_i\ |\ P_i(v)=v\rbrace$$
then $\modH_i=H_{P_i}\oplus H_{P_i}^\bot$. Then for all $v\in H_{P_1}$, $$P_2(\varphi(v))=\varphi(P_1(v))=\varphi(v),$$ 
so $\dim(H_{P_1})\leq \dim(H_{P_2}),$ 
and analogously  $\dim(H_{P_2})\leq \dim(H_{P_1})$ and so $\modH_{P_1}\cong\modH_{P_2}$ as Hilbert spaces. Similarly, for all $v\in\modH_{P_1}^\bot$ we have that 
$$P_2(\varphi(v))=\varphi(P_1(v))=\varphi(0_{\modH_1})=0_{\modH_2}, $$
so $\dim(H_{P_1}^\bot)\leq \dim(H_{P_2}^\bot)$ and analogously $\dim(H_{P_2}^\bot)\leq \dim(H_{P_1}^\bot)$ and so $\modH_{P_1}^\bot \cong\modH_{P_2}^\bot$  as Hilbert spaces. Let $U_1\colon  H_{P_1}\rightarrow H_{P_2}$ and $U_2\colon  H_{P_1}^\bot\rightarrow H_{P_2}^\bot$ be isomorphisms of Hilbert spaces and let  $U=U_1\oplus U_2$ be the extension of these maps from $\modH_{P_1}\oplus\modH_{P_1}^\bot$ to $\modH_{P_2}\oplus\modH_{P_2}^\bot$, which is again an isomorphism of Hilbert spaces. Let us prove that $U$ is an isomorphism in $IHS^P_\omega$. Indeed, let $h\in\modH_1$ and write $h=v+u$ where $v\in H_{P_1}$ and $u\in H_{P_1}^\bot$. Therefore, we have that $$U(P_1(h))=U(v)=U_1(v)=P_2(U_1(v))=P_2(U(h)),$$
and so $U$ is an isomorphism between $(\modH_1,P_1)$ and $(\modH_2,P_2)$.
\end{proof}

\begin{obs}
There are other completions of $IHS^P$, for example, those expansions where $\dim(P(\modH))=n$, and let $IHS^P_n$ denote that theory. Note that the previous argument also shows that $IHS^P_n$ satisfies the SB-property. 
\end{obs}

\subsection{Expansions with self adjoint operators}
Now we will deal with expansions of Hilbert spaces with a general self-adjoint operator. We will prove that if two such expansions are elementary bi-embeddable, then the structures are approximately unitarily equivalent (see Definition \ref{approxisom}). It remains unknown to the authors if the $SB$-property holds (see Question \ref{Q:SBadjoint} at the end of the section and the comments thereafter). We start with some background from spectral theory.

\begin{defn}
Let $\modH$ be a Hilbert space. A linear bounded operator $A\colon \modH\rightarrow\modH$ is called \textit{self-adjoint} if $A=A^*$, where $A^*$ is the \textit{adjoint operator} of $A$.
\end{defn}

\begin{defn}
A self-adjoint operator $A$ different from the zero operator is called \textit{positive}, and we write $A \geq 0$, if $\langle Ax,x\rangle\geq0$ for all $x \in \modH$. If $B$ is another self-adjoint operator and $A - B \geq 0$, then the self adjoint operator A is called \textit{greater than or equal} to the self adjoint operator B. A complex number $\lambda$ is called  an \textit{eigenvalue or punctual spectral value} of $A$ if $\dim(Ker(A - \lambda I))>0$. A complex number $\lambda$ is called a \textit{continuous spectral value} if is not an eigenvalue and the operator $A - \lambda I$ is not invertible, i.e, the bounded linear operator $(A-\lambda I)^{-1}$ cannot exist. The \textit{spectrum} of an operator $A$, denoted by $\sigma(A)$, is the set of the punctual and continuous spectral values.  
\end{defn}

We will need the following facts:

\begin{fact}[Lemma 3 section 34 \cite{lusternik1974elements}]\label{dompro}
Let $P_1,P_2$ be projections on a Hilbert space $\modH$. Then the difference $P_1-P_2$ is a projection operator if and only if $P_2$ is a part of $P_1$, i.e. $P_1P_2=P_2$.
\end{fact}

\begin{cor}\label{Comparingprojections} Let $P_1,P_2$ be projections on a Hilbert space $\modH$. If $P_2$ is a part of $P_1$, then $P_1\geq P_2$.
\end{cor}

\begin{proof}
Assume that $P_2$ is a part of $P_1$. Then by Fact \ref{dompro}, the difference $P_1-P_2$ is a projection operator and thus positive. 
\end{proof}

\begin{fact}[Theorem 3 section 31, \cite{lusternik1974elements}]\label{IntervaloEspect}
The spectrum of a self-adjoint operator $A$ lies entirely in the interval $[m,M]$ of the real axis, where $M=\displaystyle\sup_{\| x\|=1}\langle A(x),x\rangle$ and $m=\displaystyle\inf_{\| x\|=1}\langle A(x),x\rangle$.
\end{fact}

\begin{fact}[Corollary 5 section 31, \cite{lusternik1974elements}]\label{NoEmptySpec}
 Every self-adjoint operator has a non-empty spectrum.
\end{fact}

\begin{defn}
The \textit{essential spectrum} $\sigma_e(A)$ of a linear operator $A$ is the set of accumulation points of the spectrum of $A$ together with the set of eigenvalues of infinite multiplicity. The \textit{finite spectrum} $\sigma_{\text{Fin}}(A)$ of a linear operator $A$, is the complement of  $\sigma_e(A)$, is the set of isolated points of the spectrum of $A$ of finite multiplicity.
\end{defn}
We now study model-theoretically the expansion $(\modH,A)$ by adding a unary function which is interpreted as $A$ a self-adjoint operator. Note that since $A$ is bounded, it is uniformly continuous. For more details see \cite{Argoty2014model} and \cite{argoty2009hilbert}.
\begin{defn}
Let $(\modH_1,A_1)$ and $(\modH_2,A_2)$ be expansion of Hilbert spaces, where each, $A_i\colon \modH_i\rightarrow \modH_i$ is a self-adjoint operator for $i=1,2$. We say that  $U\colon  \modH_1\rightarrow\modH_2$ is an \textit{embedding} if $U$  is linear isometry and for $x\in\modH_1$ we also have
$$U(A_1(x))=A_2(U(x)). $$
We say that $U$ is an \textit{isomorphism} if $U$ is an surjective elementary embedding. 
\end{defn}

Since $A_1$ is self-adjoint, the $C^*$-algebra generated by $A_1$ is just the closure of the algebra of polynomials in $A_1$ with coefficients in $\mathbb{C}$. In particular, if $U(A_1(x))=A_2(U(x))$, then for every such polynomial $P(A_1)$, we also have $U(P(A_1)(x))=P(A_2)(U(x))$. By Corollary 4.7 \cite{argoty2013} the theory $(\modH_1,A_1)$ has quantifier elimination and thus the family of embeddings defined above are \textit{elementary embeddings}. Notice that $U$ is an \textit{isomorphism} if $U$ is an surjective embedding.

\begin{defn}\label{approxisom}
Let $(\modH_1,A_1)$ and $(\modH_2,A_2)$ be Hilbert spaces expanded with self-adjoint operators. We say that $(\modH_1,A_1)$ and $(\modH_2,A_2)$ are \textit{approximately unitarily equivalent} if there exists a sequence of isomorphisms $(U_n)_{n<\omega}$ from $\modH_1$ to $\modH_2$ such that for every $\epsilon>0$ there exists $N_\epsilon<\omega$ such that for every $n\geq N_\epsilon$, $\| A_2-U_nA_1U_n^*\|<\epsilon.$ 
\end{defn}

\begin{remark}\label{approxisompert}

Using the terminology from \textit{perturbations} in \cite{yaacov2008perturbations}, if the structures $(\modH_1,A_1)$ and $(\modH_2,A_2)$ are approximately unitarily equivalent we may also say that they are \textit{isomorphic up to perturbations of the operator.} 
\end{remark}

\begin{obs}
It is clear that if $U$ is an isomorphism between the expansion $(\modH_1,A_1)$ and $(\modH_2,A_2)$ then the expansions are approximately unitarily equivalent. However, the fact that the expansions $(\modH_1,A_1)$ and $(\modH_2,A_2)$ are approximately unitarily equivalent does not imply that they are isomorphic and also does not imply that they are elementarily bi-embeddable, see example \ref{EJAproxUnitNoIso}. 
\end{obs}

\begin{defn}
Let $(\modH_1,A_1)$ and $(\modH_2,A_2)$ be Hilbert spaces expanded with self-adjoint operators. Then $(\modH_1, A_1)$ and $(\modH_2, A_2)$ are said to be \textit{spectrally equivalent}, $(A_1 \sim_\sigma A_2)$ if all of the following conditions holds: 
\begin{itemize}
    \item[i)] $\sigma(A_1)=\sigma(A_2).$
    \item[ii)] $\sigma_e(A_1)=\sigma_e(A_2).$
    \item [iii)] $\dim\left(\lbrace x\in \modH_1\ |\ A_1(x)=\lambda x\rbrace\right)=\dim\left(\lbrace x\in \modH_2\ |\ A_2(x)=\lambda x\rbrace\right)$ for all $\lambda\in\sigma(A_1)\backslash\sigma_e(A_1).$
\end{itemize}
\end{defn}

\begin{fact}[Weyl-von Neumann-Berg Theorem, II.4.4 in \cite{davidson1996c}]\label{WeylVonNeumannBerg} Suppose $A_1$ and $A_2$ are self-adjoint operators on separable Hilbert spaces $\modH_1$ and $\modH_2$, respectively. Then $(\modH_1,A_1)$ and $(\modH_2,A_2)$ are approximately unitarily equivalent if and only if $A_1 \sim_\sigma A_2$.
 
\end{fact}

The fact above actually applies to a larger class of operators (see \cite{davidson1996c}).

\begin{ej}\label{EJAproxUnitNoIso}
Consider $(\ell_2(\N),A_1)$ and $(\ell_2(\Z),A_2)$ where 
$$A_1\left((x_n)_{n\in\N}\right)=\left(\dfrac{x_n}{n+1}\right)_{n\in \N} $$
and 
$$A_2\left((x_n)_{n\in\Z}\right)= (y_n)_{n\in\Z}\text{ where } y_n=\left\{\begin{array}{lcc} 0 & if &n\in\Z^-\\ \\ \dfrac{x_n}{n+1}  & if & n\in\Z^{\geq0}  \end{array}\right.$$
Note that $\sigma(A_1)=\lbrace0\rbrace\cup\left\lbrace\dfrac{1}{n+1}\right\rbrace_n$ where $0\in\sigma_e(A_1)$ and $\dim(Ker(A_1))=0$ while $\sigma(A_2)=\lbrace0\rbrace\cup\left\lbrace\dfrac{1}{n+1}\right\rbrace_n$ and $\dim(Ker(A_2))=\infty$, so $(\ell_2(\N),A_1)$ can't be isomorphic to $(\ell_2(\Z),A_2)$. However, by Theorem \ref{WeylVonNeumannBerg} they are approximately unitarily equivalent. Note that we can embed $(\ell_2(\N),A_1)$ in $(\ell_2(\Z),A_2)$ but not viceversa.
\end{ej}

\begin{lem}\label{HilLem1}
Let $(\modH_1,A_1)$ and $(\modH_2,A_2)$ be Hilbert spaces expanded with self-adjoint operators. If $(\modH_1,A_1)$ and $(\modH_2,A_2)$ are elementarily bi-embeddable then $A_1\sim_\sigma A_2$.
\end{lem}
\begin{proof}
Let $\varphi\colon (\modH_1,A_1)\rightarrow (\modH_2,A_2)$ and $\psi\colon (\modH_2,A_2)\rightarrow (\modH_1,A_1)$ two elementary embeddings. Let us first see that $\sigma(A_1)=\sigma(A_2).$ Set $\lambda\in\sigma(A_1)$, if $\lambda$ is an eigenvalue then there exists $x\in\modH_1$, $x\neq 0$, such that 
$A_1(x)=\lambda x$, now since 
$$A_2(\varphi(x))=\varphi(A_1(x))=\varphi(\lambda x)=\lambda\varphi(x), $$
we get that $\lambda$ is an eigenvalue of $A_2$. Note that, since $\varphi$ is an isometry, for all $x,y\in\modH_1$, $\langle x,y\rangle=\langle\varphi(x),\varphi(y)\rangle$, so if $\modH_1^\lambda$ is the eigensubpace of $\modH_1$ for the eigenvalue $\lambda$ then 
$\dim(\modH_1^\lambda)\leq\dim(\modH_2^\lambda)$. Since $\psi$ is also an elementary embedding, then $A_1$ and $A_2$ has the same eigenvalues and for all eigenvalue $\lambda\in\sigma(A_1)$, 

$$\dim(\modH_1^\lambda)=\dim(\modH_2^\lambda).\ (*)$$
\\
Now, if $\lambda\in\sigma(A_1)$ but is not an eigenvalue, there exists a sequence $\lbrace x_n\rbrace_n\subseteq\modH_1$ such that, for all $n\in\N$, $\|x_n\|=1$ and $\| A_1(x_n)-\lambda x_n\|\rightarrow0$ as $n\rightarrow\infty$. Since $\varphi$ is an isometry, $\|\varphi(x_n)\|=1$ and $\| A_2(\varphi(x_n))-\lambda \varphi(x_n)\|\rightarrow0$ as $n\rightarrow\infty$, so $\lambda\in\sigma(A_2)$. By a symmetric argument we can conclude that $\sigma(A_1)=\sigma(A_2)$ and by $(*)$, $\sigma_e(A_1)=\sigma_e(A_2)$ and  $\dim\left(\lbrace x\in \modH_1\ |\ A_1(x)=\lambda x\rbrace\right)=\dim\left(\lbrace x\in \modH_2\ |\ A_2(x)=\lambda x\rbrace\right)$ for all $\lambda\in\sigma(A_1)\backslash\sigma_e(A_1).$ So $A_1\sim_\sigma A_2$ as wanted. 
\end{proof}

Note that by Lemma \ref{HilLem1} and Fact \ref{WeylVonNeumannBerg} if $\modH_1$ and $\modH_2$ are separable Hilbert spaces, elementary bi-embeddability implies they are approximately unitarily equivalent. We will show below that if the elementary bi-embeddable condition for $(\modH_1,A_1)$ and $(\modH_2,A_2)$ holds, then the two expansions are approximately unitarily equivalent regardless if they are separable or not. 

\begin{defn}\label{DefnE_+}
Let $A$ be a self-adjoint operator. We write $E_+$ for the projection operator on $Ker(A-|A|)$, where $|A|$ is the positive square root of $A^2$ (see the proof of Theorem 36.1 in \cite{lusternik1974elements} )
\end{defn}

\begin{fact}[Theorem 36.1, \cite{lusternik1974elements}]\label{FactE_+} For every self-adjoint operator $A$, the projection operator $E_+$ defined in definition \ref{DefnE_+} satisfies the properties:
\begin{itemize}
    \item[i)] An arbitrary bounded linear operator C that commutes with $A$ also commutes with $E_+$.
    \item[ii)] $AE_+\geq  0$, $A(Id - E_+) \leq 0$, and
    \item[iii)] If $Ax = 0$, then $E_+x = x$. 
\end{itemize}

\begin{lem}\label{LemCamilo}
   Let $(\modH_1,A_1)$ and $(\modH_2,A_2)$ be Hilbert spaces expanded with positive self-adjoint operators, and let $U\colon (\modH_1,A_1) \to (\modH_2,A_2)$ be an elementary embedding. Let $S_1$ and $S_2$ be the positive square roots of $A_1$ and $A_2$ respectively. Then, for all $u\in\modH_1$, $US_1(u)=S_2U(u)$.
    \end{lem}
    \begin{proof}
   By the proof of the Theorem 9.4-2 in \cite{kreysig}, we have that for all $u\in\modH_1$:
    \[
    S_1(u)=\lim_{n\to \infty}B_n(u),
    \]
     where  the family $\lbrace B_n\rbrace_{n\in\N}$ is defined recursively by $B_0=0$ and for all $n\in \N$, 
     $$B_{n+1}=B_n+\frac{1}{2}(A_1-B_n^2).$$
     In the same way, for all $v\in\modH_2$:
    \[
    S_2(v)=\lim_{n\to \infty}C_n(v),
    \]
    where $C_0=0$ and for all $n\in \N$, $C_{n+1}=C_n+\frac{1}{2}(A_2-C_n^2)$.
   
    Therefore, for all $n\in \N^+$, there exists a polynomial $f_n(x)$ such that $B_n=f_n(A_1)$ and $C_n=f_n(A_2)$. So, for all $n\in \N^+$ and all $u\in\modH_1$ we have that $UB_n(u)=C_nU(u)$. Taking the limit as $n$ goes to infinity, we have that  $US_1(u)=S_2U(u)$.
    \end{proof}

\begin{lem}\label{HilbLem2}
Let $(\modH_1,A_1)$ and $(\modH_2,A_2)$ be Hilbert spaces expanded with self-adjoint operators, and let $E_+^1$ and $E_+^2$ be their projection operators given by Fact \ref{FactE_+} respectively. If $\varphi\colon  (\modH_1,A_1)\rightarrow (\modH_2,A_2)$ is an elementary embedding then $\varphi E_+^1\varphi^{-1}\colon Im(\varphi)\rightarrow Im(\varphi)$ is a projection operator and $\varphi E_+^1\varphi^{-1}\leq E_+^2.$ 
\end{lem}
\begin{proof}
Let $\tilde{\modH_1}=Im(\varphi)$. We will first prove that $\varphi E_+^1\varphi^{-1}:\tilde{\modH_1}\to \tilde{\modH_1}$ is self-adjoint. In order to show this, consider $y_1,y_2\in \Tilde{\modH_1}$ and choose $x_1,x_2\in\modH_1$ such that $y_1=\varphi(x_1)$ and $y_2=\varphi(x_2)$. Then, since $E_+^1$ is self-adjoint and $\varphi$ is an elementary embedding we have
$$\langle \varphi E_+^1\varphi^{-1}(y_1),y_2\rangle=\langle \varphi E_+^1(x_1),\varphi(x_2)\rangle=\langle E_+^1(x_1),x_2\rangle=\langle x_1,E_+^1(x_2)\rangle$$
$$=\langle \varphi(x_1),\varphi E_+^1(x_2)\rangle=\langle y_1,\varphi E_+^1\varphi^{-1}(y_2)\rangle,$$
and so $\varphi E_+^1\varphi^{-1}$ is a self-adjoint operator.

Now let see that $\varphi E_+^1\varphi^{-1}$ is a projection operator. Let $y\in \tilde{H_1}$, and let $x\in\modH_1$ be such that $\varphi(x)=y$, then 
$$ \varphi E_+^1\varphi^{-1}\left(\varphi E_+^1\varphi^{-1}(y)\right)=\varphi E_+^1\varphi^{-1}\left(\varphi E_+^1(x)\right)=\varphi \left((E_+^1)^2(x)\right)=\varphi\left(E_+^1(x)\right)=\varphi E_+^1\varphi^{-1}(y).$$
So $\varphi E_+^1\varphi^{-1}$ is a projection operator.

Finally, it remains to see that $\varphi E_+^1\varphi^{-1}\leq E_+^2$. In order to prove this, we will use Corollary \ref{Comparingprojections}. Since $|A_i|$ is the positive square root of $A_i^2$, for $i=1,2$, by Lemma \ref{LemCamilo} we have that for all $u\in\modH_1$, 
$$|A_2|(\varphi(u))=\varphi(|A_1|(u)).$$ 
Thus if $y\in\Tilde{\modH_1}$ is such that $y=\varphi(x)$ for $x\in\modH_1$ we get that 
$$(A_2-|A_2|)\left(\varphi E_+^1\varphi^{-1}(y)\right)=A_2\left(\varphi E_+^1\varphi^{-1}(y)\right)-\left|A_2\right|\left(\varphi E_+^1\varphi^{-1}(y)\right)$$
$$=\varphi\left(A_1\left(E_+^1(x)\right)\right)-\varphi\left(|A_1|\left(E_+^1(x)\right)\right) $$
$$=\varphi\Big(A_1\left(E_+^1(x)\right)-|A_1|\left(E_+^1(x)\right)\Big)$$
$$=\varphi\Big((A_1-|A_1|)\left(E_+^1(x)\right)\Big) =\varphi(0_{\modH_1})=0_{\modH_2}.$$
Then for every $y\in Im(\varphi)$, $\varphi E_+^1\varphi^{-1}(y)\in Ker(A_2-|A_2|)$ and so $\varphi E_+^1\varphi^{-1}\leq E_+^2$. The reader should notice that in reality $E_+^2$ has to be restricted to $Img(\varphi)=\tilde{\modH_1}$ for the two operators to act on the same space.
\end{proof}
\end{fact}

\begin{defn}\label{DefnElambda}
Let $A$ be a self-adjoint operator. For each $\lambda \in \mathbb{R}$ we write $E_\lambda$ for 
the projection operator $Id-E_+^\lambda$ where $E_+^\lambda$ is the projection operator constructed for $A-\lambda Id$ using Definition \ref{DefnE_+} (see also the proof of Theorem 36.2 in \cite{lusternik1974elements}).
\end{defn}

\begin{fact}[Theorem 36.2, \cite{lusternik1974elements}]
Let $A$ be a self-adjoint operator. Then the family of projection operators $\lbrace E_\lambda\rbrace_{\lambda}$, indexed by $\lambda\in (-\infty,\infty)$, defined as in Definition \ref{DefnElambda} satisfies the following conditions: 
\begin{itemize}
    \item[i)] If $C$ is a bounded linear operator such that $AC=CA$, then $E_\lambda C=CE_\lambda$ for all $\lambda,$
    \item[ii)] $E_\lambda\leq E_\mu$ if $\lambda<\mu$.
    \item[iii)] $E_\lambda$ is continuous on left, i.e, $\displaystyle\bigcup_{\gamma<\lambda}E_\gamma=E_\lambda$,
    \item[iv)] $E_\lambda=0$ for $-\infty<\lambda\leq m$ and $E_\lambda=Id$ for $M<\lambda<\infty$, where $m$ and $M$ are the greatest lower and least upper bounds of $A$, respectively. 
\end{itemize}

\begin{defn}
The family $\lbrace E_\lambda\rbrace_{\lambda}$ is called a \textit{decomposition of the identity} generated by $A$.
\end{defn}

\end{fact}

\begin{fact}\label{SpecTeo}[Spectral decomposition theorem of a self-adjoint operator,Theorem 36.3, \cite{lusternik1974elements}] Let $A$ be a self-adjoint operator. Then for every $\epsilon > 0,$ $A = \displaystyle\int_m^{M+\epsilon} \lambda dE_\lambda$, where the integral is to be interpreted as limit of finite sums in the sense of
uniform convergence in the space of operators, and $m = \displaystyle\inf_{\| x\|=1}\langle A(x),x\rangle$  and $M=\displaystyle\sup_{\| x\|=1}\langle A(x),x\rangle$.
 
\end{fact}

We now prove the main result of this section, a version  of the $SB$-property up to perturbations:

\begin{teo}
Let $(\modH_1,A_1)$ and $(\modH_2,A_2)$ be Hilbert spaces expanded with self-adjoint operators. If $(\modH_1,A_1)$ and $(\modH_2,A_2)$ are elementarily bi-embeddable then $(\modH_1,A_1)$ and $(\modH_2,A_2)$ are approximately unitarily equivalent.
\end{teo}
\begin{proof} 
By Fact \ref{SpecTeo} (and following the notation used in \cite{lusternik1974elements} ) for each $A_i$ and for every $\epsilon>0$ there exists $N$ and an $N$-partition $\left\lbrace[\gamma_k,\mu_k)\right\rbrace_{k=0}^{N-1}$ of $[m,M+\epsilon)$ such that $[m,M+\epsilon)=\displaystyle\bigcup_{k=0}^{N-1}[\gamma_k,\mu_k)$ where if $\Delta_k=[\gamma_k,\mu_k)$, $E^i(\Delta_k)=E^i_{\mu_k}-E^i_{\gamma_k}$ and $\nu_k\in\Delta_k$ then 
$$\left\| A_i-\displaystyle\sum_{k=0}^{N-1}\nu_kE^i(\Delta_k)\right\|<\frac{\epsilon}{2}. $$
Since $(\modH_1,A_1)$ and $(\modH_2,A_2)$ are elementarily bi-embeddable, let $\varphi\colon \modH_1\rightarrow\modH_2$ and $\psi\colon \modH_2\rightarrow\modH_1$ be two elementary embeddings. By Lemma \ref{HilbLem2}, for every $k=0,...,N-1$ we have that $\varphi E^1(\Delta_k)\varphi^{-1}\leq E^2(\Delta_k)$ and $\psi E^2(\Delta_k)\psi^{-1}\leq E^1(\Delta_k)$, then for every $k=0,...,N-1$, $\dim(H_{E^1(\Delta_k)})=\dim(H_{E^2(\Delta_k)})$,  where $H_{E^i(\Delta_k)}$ is the range of the projection operator $E^i(\Delta_k)$. Then for every $k=0,...,N-1$, $H_{E^1(\Delta_k)}\cong H_{E^2(\Delta_k)}$. Now since, $\modH_1\cong\displaystyle\bigoplus_{k=0}^{N-1} H_{E^1(\Delta_k)}$ and $\modH_2\cong\displaystyle\bigoplus_{k=0}^{N-1} H_{E^2(\Delta_k)} $, then $\modH_1$ and $\modH_2$ are isomorphic as Hilbert spaces. For $k=0,...,N-1$ let $U_k\colon  H_{E^1(\Delta_k)}\to H_{E^2(\Delta_k)}$ be an isomorphism  and let $U=\displaystyle\bigoplus_{k=0}^{N-1}U_k$ be the extension of these maps to $\modH_1$ to $\modH_2$, which is again an isomorphism of Hilbert spaces. Note that 
$$\| A_1-U^{-1}A_2U \|\leq\left\|A_1- \displaystyle\sum_{k=0}^{N-1}\nu_kE^1(\Delta_k)\right\|+\left\|\displaystyle\sum_{k=0}^{N-1}\nu_kE^1(\Delta_k)-U^{-1}A_2U\right\|,$$
and that
$$\left\|\displaystyle\sum_{k=0}^{N-1}\nu_kE^1(\Delta_k)-U^{-1}A_2U\right\|=\left\|U\displaystyle\sum_{k=0}^{N-1}\nu_kE^1(\Delta_k)-A_2U\right\|.$$
For any $h\in\modH_1\cong\displaystyle\bigoplus_{k=0}^{N-1}H_{E^1(\Delta_k)}$, we can write $h=\displaystyle\sum_{k=0}^{N-1}h_k$, where $h_k\in H_{E^1(\Delta_k)}$. Since $E^2(\Delta_k)(U(h))=U(h_k)$ we have that 
$$U\displaystyle\sum_{k=0}^{N-1}\nu_kE^1(\Delta_k)(h)=U\left(\displaystyle\sum_{k=0}^{N-1}\nu_kh_k\right)=\displaystyle\sum_{k=0}^{N-1}\nu_kU(h_k)=\sum_{k=0}^{N-1}\nu_kE^2(\Delta_k)U(h), $$
then 
$$\left\|U\displaystyle\sum_{k=0}^{N-1}\nu_kE^1(\Delta_k)-A_2U\right\|=\left\|\displaystyle\sum_{k=0}^{N-1}\nu_kE^2(\Delta_k)U-A_2U\right\|=\left\|\left(\displaystyle\sum_{k=0}^{N-1}\nu_kE^2(\Delta_k)-A_2\right)U\right\|<\dfrac{\epsilon}{2},  $$
and so $\| A_1-U^{-1}A_2U\|< \epsilon$. Thus $(\modH_1,A_1)$ and $(\modH_2,A_2)$ are approximately unitarily equivalent as desired. 
\end{proof}

In the previous theorem we got an approximate version of the $SB$-property. We do not know if the actual $SB$-property holds for $Th(\modH,A)$ and we leave it as a question:

\begin{preg}\label{Q:SBadjoint}
Let $(\modH,A)$ be Hilbert spaces expanded with self-adjoint operator and let $T=Th(\modH,A)$. Does $T$ have the $SB$-property?  
\end{preg}

In the special case when $Th(\modH,A)$ is $\aleph_0$-stable the spectrum is easy to analyze and we get a positive answer to Question \ref{Q:SBadjoint}. Recall that $\sigma(A)$ is closed subset of $\mathbb{R}$ and thus it is either countable (this is the $\aleph_0$-stable case) or has cardinality $2^{\aleph_0}$.

\begin{prop}\label{SB-Exp-SpecCont}
    Let  $(\modH,A)$ be Hilbert spaces expanded with a self-adjoint operator such that $|\sigma(A)|\leq\aleph_0$. Then $T=Th(\modH,A)$ has the SB-property.
\end{prop}
\begin{proof}
 Let $(\modH_1,A_1), (\modH_2,A_2)\models T$ be such that there exists elementary embeddings $\varphi:(\modH_1,A_1)\to(\modH_2,A_2)$ and $\psi:(\modH_2,A_2)\to(\modH_1,A_1)$. By Lemma \ref{HilLem1}, $\sigma(A_1)=\sigma(A_2)=\sigma(A)$. For each $i=1,2$ we have that 
$\modH_i=E_i\oplus E_i^\bot$, where $E_i$ is the linear closed subspace
spanned by all eigenvectors of the self-adjoint operator $A_i$ and $E_i^\bot$ is its orthogonal complement.
Since $\sigma(A)$ is a non-empty (Fact \ref{NoEmptySpec}) countable closed subset of $\mathbb{R}$, it has isolated points, so for each $i=1,2$ we have that $\dim(E_i)>0$.

\textbf{Claim} For $i=1,2$ we have $E_i=\modH_i$.

Assume otherwise. Since $E_i$ is closed under the action of the operator $A_i$, we also have that $E_i^\perp$ is closed under the action of $A_i$. Consider $(E_i^\perp,A_i\restriction_{E_i^\perp})$, it is again a Hilbert space with a self adjoint operator and $\sigma(A_i\restriction_{E_i^\perp}) \subseteq \sigma(A_i)$ and thus it is at most countable. Again by Fact \ref{IntervaloEspect} the spectrum $\sigma(A_i\restriction_{E_i^\perp})$ has isolated points and thus $(E_i^\perp,A_i\restriction_{E_i^\perp})$ has an eigenvector, a contradiction to the definition of $E_i$.

 Since $\varphi$ is an elementary map for each $\lambda$-eigenvector $x$ of $A_1$ we have that $\varphi(x)$ is also a $\lambda$-eigenvector of $A_2$. Thus, for every eigenvalue $\lambda$ of $A_1$ we get $dim(H_1^\lambda)\leq dim(H_2^\lambda)$, where $H_i^\lambda$ is the eigensubspace of $\modH_i$ associated to the eigenvalue $\lambda$. Repeating the same argument with $\psi$ we have $dim(H_2^\lambda)\leq dim(H_1^\lambda)$. Thus $A_1$ and $A_2$ have the same eigenvalues, and for all eigenvalue $\lambda$, $(H_1^\lambda,A_1\restriction_{H_1^\lambda})\cong(H_2^\lambda,A_2\restriction_{H_2^\lambda})$ and so by the Claim $(H_1,A_1)\cong (H_2,A_2)$.

\end{proof}


We end this section by considering Banach lattices, a nice example of a $\omega$-stable and $\omega$-categorical theory that fails to have the SB-property. We thank Ita\"i Ben Yaacov for pointing out the fact that the theory Banach Lattices is multidimensional.

\begin{ej}\label{Banachlattices}
Let $T$ be the theory of $L^1$-Banach lattices, which has quantifier elimination, so all embeddings are elementary embeddings. Consider the models 
$M=L^1([0,1]^{\aleph_1})$ and $N=L^1([0,1]^{\aleph_1}\sqcup [1,2])$, where $[0,1]^{\aleph_1}$ is a homogeneous probability space of density character $\aleph_1$ and $[1,2]$ is a disjoint standard Lebesgue space.
The space $M$ embeds canonically into $N$ extending each function $f\in N$ by defining it as $0$ in $[1,2]$. The map that sends $f(x)\in N$ to
$2f(x/2)$ is an embedding of $N$ into $M$. But clearly the two structures are not isomorphic.
\end{ej}

\section{Probability algebras and expansions by a generic automorphism.}\label{sec_PA}
In this section we will study the SB-property for probability algebras and their expansions with a generic automorphism. We will first show that atomless probability algebras and any completion of the theory of probability algebras have the SB-property. Then we will prove that in the expansion with a generic automorphism the SB-property fails. On the other hand, when we deal with expansions with a generic automorphism, the elementary bi-embeddability condition will give us a weaker form of the SB-property (see Definition \ref{isouppert} and Theorem \ref{SB-APrA}). We assume the reader is familiar with the model theory of 
probability measure algebras, see \cite{berenstein-Henson}, the results from \cite{song2016saturated} and its expansion with a generic automorphism \cite{berenstein2004model}. In any case, we recall some basic facts and definitions.

\begin{defn}
A \textit{probability algebra} is a pair $(\A,\mu)$ where $\A$ is a $\sigma$-complete boolean algebra and $\mu\colon \A\rightarrow[0,1]$ is a $\sigma$-additive probability measure such that $\mu(A)=0$ if and only if $A=\textbf{0}$. 
\end{defn}

\begin{defn}
A probability algebra $(\A,\mu)$ is called \textit{atomless} if for every $A\in\A\backslash\lbrace \textbf{0}\rbrace$ with $\mu(A)>0$, there exist $B\in\A\backslash\lbrace \textbf{0}\rbrace$ with $\mu(B)>0$ such that $A\cap B=B$ and $B\neq A$.
\end{defn}
We give a brief introduction of the theory of probability algebras, \textbf{Pr}, for more details and for an explicit axiomatization see \cite{ben2006schrodinger}, \cite[Section 4]{ben2010continuous} and \cite{berenstein-Henson}. Let $\modL_\textbf{Pr}=\lbrace \textbf{0},\ \textbf{1},\ \cdot^c,\ \cap,\ \cup,\ \mu\rbrace$, where $\textbf{0}$ and $\textbf{1}$ are constant symbols, $\cdot^c$ is a unary function symbol, $\cap$ and $\cup$ are binary functions symbols, and $\mu$ is an unary predicate symbol. Whenever $(\Omega,\A,\mu)$ is a probability space we will write $(\hat{\A},\mu)$ for the probability algebra associated to $(\Omega,\A,\mu)$  and we interpret the symbols $\textbf{0}, \textbf{1}, \cup$, $\cap$, $^c$ in the natural way. 
The theory of all atomless probability algebras is denoted by \textbf{APr}, it is complete and an axiomatization can be found in \cite{berenstein-Henson}.  In general, a completion $T$ of \textbf{Pr} is determined by the sizes of the atoms in any $\modA\models T$ in decreasing order and its  atomless part, more precisely: 

\begin{fact}\label{CompletionsPr}[Corollary 4.18 in \cite{berenstein-Henson}] Let $T$ be any complete $\modL_ {pr}$-theory that extends $\textbf{Pr}$ and let $\lbrace t_n\ |\ n\geq 1\rbrace$ be a sequence such that $1\geq t_1\geq t_2\geq...\geq 0$, $\displaystyle\sum_{n=1}^\infty t_n\leq 1$ and if $\modA\models T$, $\lbrace t_n\ | \ n\geq 1\text{ and } t_n>0\rbrace$ is the measure of the    $n^\text{th}$ largest atom contained in $\modA$, then:
\begin{itemize}
    \item[a)] If $\displaystyle\sum_{n=1}^\infty t_n=1$, then $T$ has a unique model, which consists of an atomic probability algebra having atoms $\lbrace A_n\ |\  n \geq  1\text{  and } t_n>0\rbrace$ such that $\mu(A_n)=t_n $ for all $n$.
    \item[b)] If $\displaystyle\sum_{n=1}^\infty t_n<1$ then the models of $T$ are exactly the probability algebras with atoms as described in $a)$ together with an atomless part of measure $1-\displaystyle\sum_{n=1}^\infty t_n$.
    \end{itemize}

\end{fact}

\begin{defn}[Definition 7.4 in \cite{berenstein-Henson}]
Let $\modB\models APr$ and $b\in\modB$, we say that 
$$\modB \! \restriction b=\lbrace a\cap b\rbrace_{a\in\modB}$$
is \textit{homogeneous} if $b\neq \emptyset$ and $\modB\!\restriction a$ has the same metric density as $\modB\! \restriction b$ for every $a\subseteq b$, $a\neq\emptyset$.
\end{defn}

\begin{notation}
For convenience, throughout the rest of this section, whenever $\modA\models APr$ and $b\in \modA$ we will write $I_b^\modA$ for the ideal $\modA\!\upharpoonright b$. If $\modA$ is clear from context, we will just write $I_{b}$.

\noindent If $\{a_i: i\geq 1\}$ are elements from $\modA$, we denote by
$\langle I^\modA_{a_i}\rangle_{i=1}^\infty$ the ideal $I_c$ where $c=\displaystyle \bigcup_{i=1}^\infty a_i$.
\end{notation}

\begin{defn}[Definition 7.5 in \cite{berenstein-Henson}]
Let $\modM=(\modB,\mu)\models APr$ and let $\mathbf{\modK^\modM}$ be the set of all infinite cardinal numbers $\kappa$ for which there
exists $b\in \modB$ such that $I_b$ is homogeneous and the density of $I_b$ is $\kappa$. For each $\kappa\in\modK^\modM$ define 
$$\Psi^\modM(\kappa):=\sup\lbrace \mu(b)\ |\ I_b \text{ is homogeneous and density of } I_b\text{ is }\kappa\rbrace. $$

We say that $b\in\modB$ is \textit{maximal homogeneous} if $I_b$ is homogeneous and $\mu(b)=\Psi^\modM(\kappa)>0$, where $\kappa
$ is the density of $I_b$. We call $(\modK^\modM,\Psi^\modM)$ \textit{Maharam invariants} for the model $\modM$.

We say $\modM$ \textit{realizes} its Maharam invariants if there exists a family $\lbrace b_\kappa\ |\ \kappa\in\modK^\modM\rbrace$ of pairwise disjoint maximal homogeneous elements of $\modB$ such that $I_{b_\kappa}$ has density $\kappa$ for every $\kappa\in\modK^\modM$ and $\sum\lbrace \mu(b_\kappa)\ |\ \kappa\in\modK^\modM\rbrace=1$.
\end{defn}

\begin{fact}[Proposition 7.10 in \cite{berenstein-Henson}]\label{RealizaMaharam} Every model $\modM \models APr$ realizes its Maharam
invariants. In particular, this means that $\modK^\modM$ is nonempty and countable.  

\end{fact}
\begin{fact}[Theorem 7.18, Maharam's theorem \cite{berenstein-Henson}]\label{MaharamMod}
Every model $\modM\models Pr$  is determined up to isomorphism by its invariants given in Fact \ref{CompletionsPr} for the atomic part and its Maharam invariants $(\modK^\modM,\Psi^\modM)$ for the atomless part.
\end{fact}

\begin{teo}\label{AprSB-prop}
$APr$ has the SB-property. 
\end{teo}

\begin{proof}
Let $\mathcal{A},\mathcal{B}\models APr$ such that there exists elementary embeddings  $\varphi\colon \mathcal{A}\to\ \mathcal{B}$ and $\psi\colon \mathcal{B} \to \mathcal{A}$. By Fact \ref{MaharamMod}, since $\modA$ and $\modB$ are atomless, they are characterized by their Maharam invariants $(\modK^\modA,\Psi^\modA)$ and $(\modK^\modB,\Psi^\modB)$ respectively. In this proof we will assume that $\modK^\modA$ and $\modK^\modB$ are countably infinite, the proof when they are finite is similar.  

Let $\modK^\modA=\lbrace \kappa_i\ |\ i<\omega\rbrace$ and for each $i<\omega$ let $\alpha_i=\Psi^\modA(k_i)$. Similarly, let $\modK^\modB=\lbrace \lambda_i\ |\ i<\omega\rbrace$ and for each $i<\omega$ let $\beta_i=\Psi^\modB(\lambda_i)$. We may order the elements in  $\modK^\modA$ so that $\kappa_{i_1}<\kappa_{i_2}$ whenever $i_1<i_2$, and similarly for $\modK^\modB$.   

In order to obtain a contradiction, assume that $\mathcal{A}\ \not\cong\ \mathcal{B}$, so $\lbrace i<\omega\ |\ (\kappa_i,\alpha_i)\neq(\lambda_i,\beta_i)\rbrace\neq\emptyset$. Take the minimum $i<\omega$ in the previous set and call it $i_0$. By Fact \ref{RealizaMaharam} there are $\lbrace a_i\rbrace_{i<\omega}\subseteq\modA$ maximal homogeneous events in $\modA$ such that $I^\modA_{a_i}$ has density $\kappa_i$  and so $\mu(a_i)=\alpha_i$. Similarly, there are $\lbrace b_i\rbrace_{i<\omega}\subseteq\modB$ maximal homogeneous events in $\modB$ such that $I^\modB_{b_i}$ has density $\lambda_i$ and $\mu(b_i)=\beta_i$.
Therefore, without loss of generality we have these two cases: 
\begin{itemize}
    \item[Case 1:\hspace{-0.5cm}]\hspace{0.5cm} $\lambda_{i_0}<\kappa_{i_0}$.
    \\
     Since $\mathcal{A}\ \hookrightarrow\ \mathcal{B}$ but $I^\modA_{a_{i_0}}\not\hookrightarrow I^\modB_{b_{i_0}}$ we must have that 
    $$\langle I^\modA_{a_i}\rangle_{i=i_0}^\infty\hookrightarrow\langle I^\modB_{b_i}\rangle_{i=i_0+1 }^\infty$$
    Then it follows that $\displaystyle\sum_{i=i_0}^\infty\alpha_i\leq\sum_{i=i_0+1}^\infty\beta_i$ which is equivalent to $\displaystyle\sum_{i=0}^{i_0-1}\alpha_i\geq\sum_{i=0}^{i_0}\beta_i$. However, by definition of $i_0$ for all $i\in\lbrace 0,...,i_0-1\rbrace$, $\alpha_i=\beta_i$, so we must have that $0\geq\beta_{i_0}$ which is a contradiction. So $\lambda_{i_0}\geq\kappa_{i_0}$. In a similar way we can show that $\kappa_{i_0}\geq\lambda_{i_0}$ and we have that $\lambda_{i_0}=\kappa_{i_0}$, so this case is not possible.
    \item[Case 2:\hspace{-0.5cm}]\hspace{0.5cm} $\alpha_{i_0}<\beta_{i_0}.$
    \\
    By the previous case, we know that  $\lambda_{i_0}=\kappa_{i_0}$. By hypothesis we have that  $\alpha_{i_0}<\beta_{i_0}$ and since $$1=\displaystyle\sum_{i<\omega}\alpha_i=\sum_{i<\omega}\beta_i\text{ and } \displaystyle\sum_{i<i_0}\alpha_i=\displaystyle\sum_{i<i_0}\beta_i$$
     we also have that
    $$\displaystyle\sum_{i=i_0+1}^\infty\alpha_i>\sum_{i=i_0+1}^\infty\beta_i.$$
    However, as $I^\modA_{a_{i_0+1}}\not\hookrightarrow I^\modB_{b_{i_0}}$  and $\modA\ \hookrightarrow\ \modB$ we must have
    $$\displaystyle\langle I^\modA_{a_i}\rangle_{i=i_0+1}^\infty\hookrightarrow\langle I^\modB_{b_i}\rangle_{i=i_0+1}^\infty $$
    which means that $\displaystyle\sum_{i=i_0+1}^\infty\alpha_i\leq\sum_{i=i_0+1}^\infty\beta_i$ and we have a contradiction, therefore $\alpha_{i_0}\geq\beta_{i_0}$. In a similar way we have $\beta_{i_0}\geq\alpha_{i_0}$, so $\alpha_{i_0}=\beta_{i_0}$ so this case is neither possible.
    
\end{itemize}
In this way, we have shown that $\lbrace i<\omega\ |\ (\kappa_i,\alpha_i)\neq(\lambda_i,\beta_i)\rbrace=\emptyset$ thus by Fact \ref{MaharamMod}, $\mathcal{A}\cong\mathcal{B}.$
\end{proof}

\begin{cor}
Any completion of $Pr$ has the SB-property.
\end{cor}
\begin{proof}
Let $T$ be a completion of $Pr$ and let $\mathcal{A},\mathcal{B}\models T$ such that there exists elementary embeddings  $\varphi\colon \mathcal{A}\to\ \mathcal{B}$ and $\psi\colon \mathcal{B} \to \mathcal{A}$. By Fact \ref{CompletionsPr} and Theorem \ref{AprSB-prop} we may assume that $\modA$ and $\modB$  are not atomless. So, without loss of generality, by Fact \ref{MaharamMod}, $\modA$ is determined by the set $\{t_i\}_{i=1}^\infty$, where $1\geq t_1\geq t_2\geq...>0$, for it atomic part, and its Maharam invariants $(\modK^\modA,\Psi^\modA)$ for it atomless part. Analogously, $\modB$ is determined by the set $\{\tau_i\}_{i=1}^\infty$, where $1\geq \tau_1\geq \tau_2\geq...>0$, for it atomic part, and its Maharam invariants $(\modK^\modB,\Psi^\modB)$ for it atomless part.

Since being an atom is an elementary property, the collection of atoms of size $r>0$ in $\modA$ go through the elementary map $\varphi$ to the collection of atoms of size $r>0$ in $\modB$, so we would have that $\lbrace t_i\rbrace_{i=1}^\infty\subseteq\lbrace \tau_i\rbrace_{i=1}^\infty$, and similarly the collection of atoms of size $s>0$ in $\modB$ go through the elementary map $\psi$ to the collection of atoms of size $s>0$ in $\modA$, so $\{\tau_i\}_{i=1}^\infty\subseteq\{t_i\}_{i=1}^\infty$, thus $\lbrace t_i\rbrace_{i=1}^\infty=\lbrace \tau_i\rbrace_{i=1}^\infty$. Also, by Theorem \ref{AprSB-prop} we know that $(\modK^\modA,\Psi^A)=(\modK^\modB,\Psi^B)$, then by Theorem \ref{MaharamMod} we have that $\modA\cong\modB.$

\end{proof}
Now we will deal with atomless probability algebras expanded with a generic automorphism, see \cite{BenYac-Berens-2008perturbations}. Recall that an automorphism $T\colon \modA\to \modA$ is generic if and only if for every $\epsilon>0$ and every $n\geq 1$ there is $c\in \modA$ such that $c,T(c),\dots, T^{n-1}(c)$ are disjoint and $\mu(c\cup T(c)\cup \dots \cup T^{n-1}(c))\geq 1-\epsilon$. We call this theory $APrA$. Details about this theory can be found in \cite[section 18]{ContModelTheory} and in \cite{berenstein2004model}. We will use in the rest of this section the fact that $APrA$ has quantifier elimination.

Sometimes it is convenient to work with probability spaces instead of probability algebras. If $(\Omega,\A,\mu)$ is a probability space with associated probability algebra $\hat \A=\modA$ and $T'\colon \Omega\to \Omega$ is an invertible measure preserving transformation such that the induced map $\hat{T'}\colon \modA\to\modA$ coincides with $T$, then $T$ is generic if and only if $T'$ is aperiodic, that is, for every $n\geq 1$, $\mu(\{\omega\in \Omega\ |\  {T'}^n(\omega)=\omega\})=0$. Aperiodic maps include, among other, ergodic maps. 
\\
\\
Let us recall some terminology from \cite{Wal}.

\begin{defn}[Definition 2.6 in \cite{Wal}]
    Suppose $(\Omega_1,\A_1,\mu_1)$ and $(\Omega_2,\A_2,\mu_2)$ are probability spaces and $T_1: \Omega_1\to\Omega_1$ and $T_2:\Omega_2\to\Omega_2$ are invertible measure preserving maps. We say that $T_2$ is a \textit{factor} of $T_1$ if there exist $B_i\in\A_1$ with $\mu_i(B_i)=1$ and $T_i(B_i)\subseteq B_i$ and there exist a measure preserving transformation $\phi:B_1\to B_2$ with $\phi T_1(\omega)=T_2\phi(\omega)$, for all $\omega\in B_1$.
\end{defn}

\begin{obs}\label{factorembe}
    In the definition above, let $\modA_1$ and $\modA_2$ be the associated measure algebras of the probability spaces and $\hat{T_1}$, $\hat{T_2}$ are the associated automorphisms. Then, if $T_2$ is a factor of $T_1$, we obtain by quantifier elimination in $APrA$ that $(\modA_1,\hat{T}_1)$ is elementary embeddable in $(\modA_2,\hat{T}_2)$.
\end{obs}

We will also need the following notion introduced by Y. Sinai:

\begin{defn}[Definition 2.6 in \cite{Wal}]
    Suppose $(\Omega_1,\A_1,\mu_1)$ and $(\Omega_2,\A_2,\mu_2)$ are probability spaces and $T_1: \Omega_1\to\Omega_1$ and $T_2:\Omega_2\to\Omega_2$ are invertible measure preserving maps. We say that $T_1$ and $T_2$ are \emph{weakly isomorphic} if $T_2$ is a factor of $T_1$ and $T_1$ is a \textit{factor} of $T_2$.
\end{defn}

Following the notation in observation \ref{factorembe}, if $T_1$ and $T_2$ are weakly isomorphic then the
structures $(\modA_1,\hat{T}_1)$ and $(\modA_2,\hat{T}_2)$ are elementary bi-embeddable. So, the SB-property in $APrA$ corresponds to weak isomorphism agreeing with isomorphim for aperiodic maps. This question is now well understood in ergodic theory: 

\begin{fact}\label{F:Rudolph}
    The theory $APrA$ does not have the SB-property for separable models.
\end{fact}

\begin{proof}
By \cite[Theorem 4.32 part v.]{Wal} (and for the proof see the original papers \cite{Polit, Rudolph}) weak isomorphism does not agree with isomorphism. The maps involved are Kolmogorov automorphims on standard Lebesgue spaces, which are ergodic and thus aperiodic.
\end{proof}

We will now study the SB-property up to perturbations. In order to do so, we need to measure the distance between two measure preserving transformations. For $(\modA,T), (\modA,S)\models APrA$, we define $$d(T,S)=\sup_{a\in \modA}\mu\{T(a)\triangle S(a)\}$$
 See \cite[pp. 69-73]{halmos1976measure} for basic properties of this metric.

\begin{defn}\label{isouppert} Let $(\modA,S), (\modB,T)\models APrA$. We say that $(\modA,S)$ and $(\modB,T)$ are \textit{approximately isomorphic} if there exists a sequence of probability algebras isomorphisms $(\varphi_n)_{n<\omega}$ from $\modA$ to $\modB$ such that for every $\epsilon>0$ there exists $N_\epsilon<\omega$ such that for every $n\geq N_\epsilon$, $d(T,\phi_n^{-1}S\phi_n)<\epsilon$. Using the terminology from perturbations from \cite{yaacov2008perturbations}, we may also say that the structures $(\modA,S)$ and $(\modB,T)$ are \textit{isomorphic up to perturbations of the automorphism.} \end{defn}

Before proving the main result, we will need a basic fact:

 \begin{lem}\label{ideals}
Let $b\in \modA$ with $\mu(b)>0$ and $T(b)=b$. Then, after rescaling the measure of the space, $T$ restricted to the ideal $\hat \A \!\!\upharpoonright_b=\{a\in \hat \A\ |\ a\subseteq b\}$ is again generic, i.e. for any integer $n\geq 1$ and $\epsilon>0$ there is $c_b\in \modA\!\!\upharpoonright_b$ such that $c_b,T(c_b),\dots, T^{n-1}(c_b)$ are disjoint and $\mu(c_b\cup T(c_b)\cup \dots \cup T^{n-1}(c_b))\geq \mu(b)-\epsilon$.
\end{lem}

\begin{proof}

Since $T(b)=b$, the map $T$ acts on the ideal $\modA \!\!\upharpoonright_b$. Let $n\geq 1$ be an integer and let $\epsilon>0$. Since $T$ is generic, there is $c\in \modA$ such that $c,T(c),\dots, T^{n-1}(c)$ are disjoint and $\mu(c\cup T(c)\cup \dots \cup T^{n-1}(c))\geq 1-\epsilon$. Let $d=\textbf{1}\setminus (c\cup T(c)\cup \dots \cup T^{n-1}(c))$, so $\mu(d)<\epsilon$
and the set $\{d,c,T(c),\dots, T^{n-1}(c)\}$ forms a measurable partition of $\textbf{1}$. Now let $c_b=c\cap b$, note that $T^i(c_b)=T^i(c)\cap b\in \modA \!\!\upharpoonright_b$ and that $\{d\cap b,c_b,T(c_b),\dots, T^{n-1}(c_b)\}$ forms a measurable partition of $b$.
Also note that $\mu(d\cap b)\leq \mu(d)<\epsilon$, so $c_b$ is the desired witness for the property. \end{proof}

\begin{teo}\label{SB-APrA}
Let $(\modA,T)$ and $(\modB,S)\models APrA$ be such that $(\modA,T)$ and $(\modB,S)$ are elementary bi-embeddable. Then $(\modA,T)$ and $(\modB,S)$ are approximately isomorphic.
 \end{teo}

\begin{proof}
Let $(\modA,T), (\modB,S)\models APrA$ and assume there are elementary embeddings $\varphi\colon (\modA,T) \to (\modB,S)$, $\psi\colon  (\modB,S) \to (\modA,T)$. By the previous theorem, the structures $\modA$ and $\modB$
are isomorphic and without loss of generality, we may assume $\mathcal{A}=\mathcal{B}$. Let $(\modK,\Psi)$ be the Maharam invariants and this structure and for each $\kappa_i\in \modK$, let $b_i$ be a maximal homogeneous element of density character $\kappa_i$, so $\mu(b_i)=\alpha_i>0$.

Since $T$ is an isomorphism, it has to send $b_i$ to another maximal homogeneous element of density character $\kappa_i$, so $T(b_i)=b_i$. Thus we may decompose $T$ as a direct sum of the family of restrictions $T\upharpoonrightleft \!\! I_{b_i}$ on the corresponding homogeneous ideals. 

By Lemma \ref{ideals} for each $i$, $T_i=T\upharpoonrightleft \!\! I_{b_i}\to I_{b_i}$ is again generic. Let $\epsilon>0$ and let $n\geq 1$. Then there is $c_i\in  I_{b_i}$ such that $c_i,T(c_i),\dots, T^{n-1}(c_i)$ are disjoint and $\mu(c_i\cup T(c_i)\cup \dots \cup T^{n-1}(c_i))\geq \alpha_i-\dfrac{\epsilon}{2^i}$. Note that since $c_i\subseteq b_i$, $c_i$ is homogeneous of density character $\kappa_i$. Similarly, there is $c_i'\in  I_{b_i}$ such that $c_i',S(c_i'),\dots, S^{n-1}(c_i')$ are disjoint and $\mu(c_i'\cup S(c_i')\cup \dots \cup S^{n-1}(c_i'))\geq \alpha_i-\dfrac{\epsilon}{2^i}$. As before, $c_i'$ is homogeneous of density character $\kappa_i$. By choosing 
$\min\{\mu(c_i),\mu(c_i')\}$ and taking a subset of one of the two sets if necessary, we may assume that $\mu(c_i)=\mu(c_i')$. 

We now define an automorphism $\phi$ of $I_{b_i}$. We follow a standard argument about almost conjugacy of two aperiodic maps (see for example the section on uniform approximation in \cite{halmos2017lectures}) known to hold on separable models.
By \cite[Proposition 3.2]{song2016saturated} the measure algebra $I_{b_i}$ is strongly $\kappa_i$-homogeneous. We first construct $n+1$ automorphisms  $\{\phi_j\colon I_{b_i}\to I_{b_i}:0\leq j\leq n\}$. Since $\mu(c_i)=\mu(c_i')$ there is $\phi_0\colon I_{b_i}\to I_{b_i}$ such that $\phi_0(c_i)=c_i'$. Now for $0<j< n$, let $\phi_j\colon I_{b_i}\to I_{b_i}$ be defined by $\phi_j(a)=S^{j}\left(\ \phi_0\left(T^{-j}(a)\right)\ \right)$. Note that for each $j< n$, $\phi_j(T^{j}(c_i))=S^{j}(c_i')$. Since $\mu(c_i)=\mu(c_i')$ we now get 
$$\mu(c_i\cup T(c_i)\cup \dots \cup T^{n-1}(c_i))=\mu(c_i'\cup S(c_i')\cup \dots \cup S^{n-1}(c_i')).$$ 
Again using \cite[Proposition 3.2]{song2016saturated} there is an automorphisms $\phi_n\colon I_{b_i}\to I_{b_i}$ that sends $\displaystyle\bigcup_{j\leq n-1}T^j(c_i)$ to $\displaystyle\bigcup_{j\leq n-1} S^j(c_i')$ and thus it also sends $b_i\setminus \displaystyle\bigcup_{j\leq n-1}T^j(c_i)$ to $b_i\setminus \displaystyle\bigcup_{j\leq n-1} S^j(c_i')$. 
\\
\\
We define $\phi\colon I_{b_i}\to I_{b_i}$ "piecewise" using the family $\{\phi_{j}\}_{j=0}^n$. If $a\subseteq \displaystyle\bigcup_{j\leq n-1} T^j(c_i)$, then $\phi(a)=\displaystyle\bigcup_{j\leq n-1} \phi_j(a\cap T^{j-1}(c_j))$, and if $a\subseteq \left(b_i\setminus \displaystyle\bigcup_{j\leq n-1}T^j(c_i)\right)$, then $\phi(a)=\phi_n(a)$. And in general, for $a\subseteq b_i$, $$\phi(a)=\phi\left(a\cap \displaystyle\bigcup_{j\leq n-1} T^j(c_i)\right)\cup \phi\left(a\cap \left(b_i \setminus(\displaystyle\bigcup_{j\leq n-1} T^j(c_i))\right) \right).$$

Note that for $a\subseteq \displaystyle\bigcup_{j\leq n-2}T^j(c_i)$ we have that $\phi T(a)=S \phi (a)$ and thus $$d(\phi T,S\phi)\leq  \dfrac{\alpha_i}{n}+ \dfrac{\epsilon}{2^i}.$$

We put together the maps defined for all indexes $i$ and build $\phi\colon \modA \to \modB $ (that only depends on $n$ and $\epsilon$) by taking the direct sum of the maps $\phi$ defined on each $I_{b_i}$. Then 
$d(\phi T \phi^{-1},S)=d(\phi T,S\phi)\leq  \displaystyle\sum_{1\leq i}\left(\frac{\alpha_i}{n}+ \frac{\epsilon}{2^i}\right)=\frac{1}{n}+ \epsilon$.
Since we can choose $n$ and $\epsilon $ arbitrarily, then $(\modA,T)$ and $(\modB,S)$ are isomorphic up to perturbations of the automorphism.

\end{proof}

In the language of perturbations, Theorem \ref{SB-APrA} shows that $APrA$ has the SB-property up to perturbations of the automorphism. As mentioned in the introduction the theory $APrA$ is stable, not superstable and $\aleph_0$-stable up to perturbations (see [10]). By Fact \ref{superstable-SB} and its continuous counterpart Theorem \ref{TeoSuperstableSB} it does not have the SB-property (by Fact \ref{F:Rudolph} the $SB$-property even fails for separable models). This seems to point out an interesting fact around the stability spectrum when we see it up to perturbation, we may recover the desired properties (in this case isomorphism) up to perturbation. 

\begin{preg}
  Assume $T$ is $\aleph_0$-stable up to perturbations of some predicates and functions. Also assume $T$ is non-multidimensional. Does $T$ have the $SB$-property up to perturbations?
\end{preg}

\section{Randomizations.}\label{sec_Rand}
In this section we will study the SB-property in Randomizations. We will show that a first order theory $T$ with $\leq\omega$ countable models has the SB-property for countable models if and only if $T^R$ has the SB-property for separable models. We also use randomizations to construct a continuous complete without the SB-property. We assume the reader is familiar with the model theory of 
randomizations, see  \cite[Section 5]{ andrews2019independence}, \cite{yaacov2013theories} and \cite{ben2009randomizations}. In any case, we recall some basic facts and definitions.
\\
\\
Given a first order language $\modL$, define $\modL^R$, the associated \textit{randomization language}, which is a continuous language of two sorts  $(\textbf{K},\B)$, where $\textbf{K}$ is a sort of random variables an $\B$ is a sort of events. The language consists of a function $\llbracket \varphi(\cdot)\rrbracket \colon \textbf{K}^n\longrightarrow \B$ for each formula  $\varphi$ in $\modL$ with $n$ variables, and Boolean operations $\top$, $\bot$, $\sqcup$, $\sqcap$,$\neg$ in sort $\B$ and an unary predicate $\mu\colon  \B\longrightarrow [0,1].$ We will use letters $f,g,h,...$ for elements of the sort $\textbf{K}$ and letters $A,B,C,...$ for elements inside the sort $ \B.$
\begin{defn}[Definition 5.1.1. in \cite{andrews2019independence}]
Let $\modM$ be a first order $\modL$-structure, define a \textit{randomization of $M$}, as a pre-structure $(\textbf{K},\B)$ on the language $\modL^R$ such that:
\begin{itemize}
    \item[1.] $(\Omega,\B,\mu)$ is a probability space, with $\sigma$-algebra $\B$ and measure $\mu$. 
    \item[2.] $\textbf{K}$ is a set of functions $f \colon \Omega\longrightarrow \modM$. Some times $\textbf{K}$ could be denoted by $\textbf{K}_\modM$ or $\modM^R.$ 
    \item[3.] For each $\modL-$formula $\varphi(x)$ and each $n-$tuple $\bar{f}\in \textbf{K}^n$
     $$\llbracket\varphi(\bar{f})\rrbracket=\lbrace\  t\in\Omega\ |\ \modM\models\varphi\left(\bar{f}(t)\right)\ \rbrace.  $$
     \item[4.]For each $B\in\B$ and all $\epsilon>0$ there exists $f,g\in\textbf{K}$ such that 
     $$\mu(\llbracket f=g\rrbracket\triangle B)<\epsilon.$$
     i.e, any event can be approximated with an equality of two functions.
     \item[5.] For all $\modL-$formulas $\varphi(x,y)$, all $\epsilon>0$ and all $\overline{g}\in\textbf{K}^n$ there exists $f\in\textbf{K}$ such that $$\mu(\llbracket \varphi(f,\overline{g})\rrbracket\triangle \llbracket\exists x\varphi(x,\overline{g})\rrbracket)<\epsilon$$
     i.e, we have existence of approximate witnesses\hspace{0.1cm} for existentials on $\modL-$formulas.
     \item[6.] For all $f,g\in\textbf{K}$ define $d_\textbf{K}(f,g)=\mu(\llbracket f\neq g\rrbracket).$
     \item[7.] For all $A,B\in\B$ define $d_\B(A,B)=\mu(A\triangle B).$
\end{itemize}
Here $d_\textbf{K}$ and $d_\B$ are pseudo-metrics. We can take a quotient identifying elements that are at distance zero from each other for obtain a metric space. Since every metric space has a complete extension, from the quotient we can obtain a complete metric structure associated to $(\textbf{K},\B)$ called the \textit{completion of} $(\textbf{K},\B)$.
\end{defn}

\begin{defn}[Definition 5.1.2. in \cite{andrews2019independence}]
For each first order theory $T$, the \textit{randomized theory} $T^R$ is the set  of sentences, in the continuous sense, that are true in all the randomizations of  models of $T$.
\end{defn}
Let $\modB$ be the set of Borel sets in $[0,1)$ and let $([0,1),\modB,\lambda)$ be the usual probability space where $\lambda$ is the restriction of Lebesgue measure to $\modB$. For a first order theory $T$ and $\modM\models T$ define:
$$\modM^{[0,1)}=\lbrace\ f\colon [0,1)\rightarrow M\ |\ \left|\text{Im}(f)\right|\leq\aleph_0\text{ and } f^{-1}[m]\in\modB,\text{ for all } m\in M\ \rbrace.$$
\begin{defn}
\textit{(Definition 2.1 in \cite{andrews2015separable})} The \textit{Borel randomization of $\modM$} is the pre-structure $(\modM^{[0,1)},\modB)$ for $\modL^R$ whose universes for the sorts $\textbf{K}$ and $\B$ are $\modM^{[0,1)}$ and $\modB$ respectively, whose measure $\mu$ is given by $\mu(B)=\lambda(B)$ for all $B\in\modB$ and for each $n$-ary $\modL$-formula $\varphi$, the function $\llbracket\varphi(\cdot)\rrbracket \colon \textbf{K}^n\rightarrow\B$ is interpreted as:
$$\llbracket\varphi\left( \bar{f} \right)\rrbracket=\lbrace\ t\in[0,1)\ |\ \modM\models\varphi\left( \bar{f}(t) \right)\ \rbrace $$
and its distance predicates are defined by $$d_\B(B,C)=\mu(B\triangle C),\ d_\textbf{K}=\mu(\llbracket f\neq g\rrbracket), $$
where $\triangle$ is the symmetric difference operation. 
\end{defn}
We now lists some facts that we will need later in this paper to prove Theorem \ref{GranTeorema} and Corollary \ref{Cor:SBrand}. The
presentation here is largely taken from \cite{andrews2015separable} and \cite{ben2009randomizations}.

\begin{fact}\label{ContableTSepTR}
\textit{(Lemma 4.2 in \cite{andrews2015separable})} A model $\modM$ of $T$ is countable if and only if the Borel randomization $(\modM^{[0,1)},\modB)$ is separable.
\end{fact}
\begin{defn}
\textit{(Definition 4.1 in \cite{andrews2015separable})} A pre-model of $T^R$ is called \textit{strongly separable} if is elementarily embeddable in $(\modM^{[0,1)},\modB)$ for some countable model $\modM$ of $T$. 
\end{defn}
\begin{fact}
\textit{(Proposition 4.3 in \cite{andrews2015separable})} A pre-model of $T^R$ is strongly separable if and only if is separable and elementarily embeddable in the Borel randomization of some model of $T$.
\end{fact}

\begin{defn}
\textit{(Definition 7.1 in \cite{andrews2015separable})} Let $M\models T$, let $I$ be a finite or countable set of indexes, let $[0,1)=\displaystyle\bigcup_{i\in I}B_i$ be a partition of $[0,1)$ into Borel sets, and for each $i\in I$ let $\modM_i\ \sbestrucel\ \modM$. We define
$$\prod_{i\in I }\modM_i^{B_i}=\lbrace f\in\modM^{[0,1)}\ |\ (\forall i\in I)\ \lambda(\llbracket f\in  M_i\rrbracket\triangle B_i)=0\ \rbrace. $$
$\displaystyle\left(\prod_{i\in I}\modM_i^{B_i},\modB\right)$ is a pre-strcuture and $\displaystyle\left(\prod_{i\in I}\modM_i^{B_i},\modB\right)\subseteq (\modM^{[0,1)},\modB)$. We call $\displaystyle\left(\prod_{i\in I}\modM_i^{B_i},\modB\right)$ a \textit{product randomization in $\modM$}.
\end{defn}

\begin{obs}\hspace{0.01cm} 
\begin{itemize} 
    \item[i)] We may view $\displaystyle\left(\prod_{i\in I}\modM_i^{B_i},\modB\right)$ as the result of sampling from $\modM_i$ with probability $\lambda(B_i)$ for each $i \in I$.
    \item[ii)] We allow the possibility that some of the sets $B_i$ are empty.
\end{itemize}
\end{obs}

\begin{fact}
\textit{(Theorem 7.3 in \cite{andrews2015separable})} Every product randomization in $\modM$ is a pre-complete elementary sub-structure of the Borel randomization $(\modM^{[0,1)},\modB)$.
\end{fact}
\begin{fact}\label{RandSepNoIsoCaract}
\textit{(Theorem 7.5 in \cite{andrews2015separable})} Let $I$ be a finite or countable set of indexes, let $\modM_i\ \sbestrucel\ \modM$ for each $i\in I$, and let $[0,1)=\displaystyle\bigcup_{i\in I}B_i$ and $[0,1)=\displaystyle\bigcup_{i\in I}C_i$ be two partitions of $[0,1)$ into Borel sets. Suppose that $\lambda(B_i)=\lambda(C_i)$ for each $i\in I$. Then the product randomizations $\displaystyle\left(\prod_{i\in I}\modM_i^{B_i},\modB\right)$ and $\displaystyle\left(\prod_{i\in I}\modM_i^{C_i},\modB\right)$ are isomorphic. \end{fact}
\begin{remark}
A first order theory $T$ \textbf{has $\leq\omega$ countable models} if there is a finite or countable set $S$ of countable models of $T$ such that every countable model of $T$ is isomorphic to some member of $S$.  If $T$ has $\leq\omega$ countable models, $T$ has a countable saturated model.
\end{remark}

\begin{fact}
\textit{(Proposition 5.7 in \cite{andrews2015separable})} The following are equivalent:
\begin{itemize}
    \item[i)] $T$ has a countable saturated model. 
    \item[ii)] $T^R$ has a separable $\omega$-saturated model.
    \item[iii)] Every separable pre-model of $T^R$ is strongly separable.
\end{itemize}
\end{fact}

\begin{fact}\label{CaractSepModel}
\textit{(Lemma 8.4 in \cite{andrews2015separable})} Suppose $T$ has $\leq\omega$ countable models, and let $\modM$ be a countable saturated model of $T$. Then every separable model $(\textbf{K},\B)$ of $T^R$ is isomorphic to a product randomization $\left(\displaystyle\prod_{i\in I}\modM_i^{B_i},\modB\right)$ in $\modM$. Moreover, the models $\modM_i$ can be taken to be pairwise non-isomorphic.
\end{fact}

\begin{defn}
Let $T$ be a theory with $\leq\omega$ countable models, let $\modM(T)$ be the countable saturated model of $T$ and let $I(T)$ be the set of all isomorphism types of elementary submodels of $\modM(T)$. A \textit{density function on $I(T)$} is a function $\rho \colon  I(T)\rightarrow [0,1]$ such that $\displaystyle\sum_{i\in I(T)}\rho(i)=1$. 
\end{defn}
\begin{defn}
Assume $T$ has $\leq\omega$ countable models and let $(\textbf{K},\B)$ be a separable model of $T^R$. A \textit{density function for $(\textbf{K},\B)$} is a function density function $\rho$ on $I(T)$ such that $(\textbf{K},\B)$ is isomorphic to some product randomization 
$$\left(\prod_{i\in I(T)}\modM_i^{B_i},\modB\right) $$
in $\modM(T)$ where $\modM_i$ has isomorphism type $i$ and $\lambda(B_i)=\rho(i)$ for each $i\in I(T)$.
\end{defn}
\begin{fact}
\textit{(Theorem 8.6 in \cite{andrews2015separable})} Suppose that $T$ has $\leq\omega$ countable models. Then
\begin{itemize}
    \item[i)] every separable model of $T^R$ has a unique density function;
    \item[ii)] any two separable models of $T^R$ with the same density function are isomorphic; 
    \item[iii)] for every density function $\rho$ on $I(T)$, there is a separable model $(\textbf{K},\B)$ on $T^R$ with density function $\rho$.
\end{itemize}
\end{fact}

\begin{fact}\label{UniqnessTheoremSep}
\textit{(Theorem 8.8 in \cite{andrews2015separable})} Let $\modM$ be a countable model of $T$ and let $I$ be finite or countable set of indexes. Suppose that
\begin{itemize}
    \item[i)] $\modM_i\ \sbestrucel\ \modM$ for each $i\in I$;
    \item[ii)] the models $\modM_i$ are pairwise non-isomorphic;
    \item[iii)] $[0,1)=\displaystyle\bigcup_{i\in I}B_i$ and $[0,1)=\displaystyle\bigcup_{i\in I}C_i$ are partitions of $[0,1)$ into Borel sets;
    \item[iv)] $\displaystyle\left(\prod_{i\in I}\modM_i^{B_i},\modB\right)$ and $\displaystyle\left(\prod_{i\in I}\modM_i^{C_i},\modB\right)$ are isomorphic.
\end{itemize}
Then $\lambda(B_i)=\lambda(C_i)$ for each $i\in I$. 
\end{fact}

We are now ready to prove Theorem \ref{GranTeorema}, we just need the following well known result on extending partial orders.

\begin{fact}[Szpilrajn extension theorem]  \label{LemaOrdenes} Any partial order $(P,\sbestrucel)$, can be extended to a total order $(P,\leq)$. 
\end{fact}

\begin{teo}\label{GranTeorema}
Let $T$ be a first order complete theory with the SB-property for countable models. Furthermore, suppose that $T$ has at most countably many countable models. Then $T^R$ has the SB-property for separable models, that is, if $\modN_1,\ \modN_2\models T^R$ are separable models and there exists elementary embeddings  $\varphi \colon \modN_1 \to \modN_2$ and $\psi\colon \modN_2\to\modN_1$, then $\modN_1\cong\modN_2$.
\end{teo}
\begin{proof}
Let $\modU$ be the set of countable models of $T$ and let $\modM,\modN\models T$. We write $\modM\ \widetilde{\leq}\ \modN$ if there exist an elementary embedding $\varphi \colon \modM\to\modN$. Since $T$ has the SB-property,  $(\modU, \widetilde{\leq})$ is a countable partial order. By Fact \ref{LemaOrdenes} we can extend $(\modU, \widetilde{\leq})$ to a countable total order $(\modU, \leq_\modU)$ and since $(\Q,\leq)$ is universal for countable total orders, we can find $\Sigma\subseteq [0,1]\cap \Q$ such that  $( \modU,\leq_\modU)\cong(\Sigma,\leq) $. Whenever $i\in\Sigma$ we write $\modM_i$ for the model in $\modU$ corresponding to the previous isomorphism. Since $T$ has countably many models, it has a prime model and a saturated model. Thus, $\Sigma$ has a minimum and a maximum element. Thus we can choose $\Sigma$ such that, $0,1  \in\Sigma$ and the following holds:
\begin{itemize}
    \item[i)] $\modM_0$ is the prime model of $T$.
    \item[ii)] $\modM_1$ is the countable saturated model of $T$.
    \item[iii)] If $i,j\in\Sigma$ and $\modM_i\ \widetilde{\leq}\ \modM_j$ then $i\leq j$.
\end{itemize}
Let $\modN_1,\modN_2\models T^R$ be separable elementary bi-embeddable models. By Fact \ref{CaractSepModel} we can assume that  $\displaystyle\modN_1\cong\prod_{i\in \Sigma}\modM_i^{A_i}$ and $\displaystyle\modN_2\cong\prod_{i\in\Sigma}\modM_i^{B_i}$, for some Borel partitions $\lbrace A_i\rbrace_{i\in\Sigma}$ and $\lbrace B_i\rbrace_{i\in\Sigma}$ of $[0,1)$.
\\
\\
In order to get a contradiction, assume that  $\modN_1\ncong\modN_2$. By Facts \ref{RandSepNoIsoCaract} and \ref{UniqnessTheoremSep}
$$K=\left\lbrace\ k\in\Sigma\ |\  \lambda(A_k)\neq\lambda(B_k)\ \right\rbrace\neq\emptyset,$$
and let $k_0=\inf(K)$. We need to consider the following cases: 
\begin{itemize}
    \item[\textit{Case 1:}\hspace{-0.5cm}] \hspace{0.5cm} $k_0$ is the minimum element of $ K.$  
    \\
    Since $\lambda(A_{k_0})\neq\lambda(B_{k_0})$, without loss of generality we may assume that $\lambda(A_{k_0})<\lambda(B_{k_0}).$ So $\displaystyle\sum_{i\in\Sigma,\ k_0<i\leq1}\lambda(B_i)\ <\ \displaystyle\sum_{i\in\Sigma,\ k_0<i\leq1}\lambda(A_i)$. However, since $\modN_1\ \hookrightarrow\ \modN_2$ and for all $i>k_0$, $\modM_i\not\hookrightarrow\modM_{k_0}$, we must have that
    $$\prod_{i\in\Sigma,\ k_0<i\leq1}\modM_i^{A_i}\hookrightarrow\prod_{i\in\Sigma,\ k_0<i\leq1}\modM_i^{B_i} $$
    and then $\displaystyle\sum_{i\in\Sigma,\ k_0<i\leq1}\lambda(A_i)\leq\displaystyle\sum_{i\in\Sigma,\ k_0<i\leq1}\lambda(B_i)$, a
    contradiction. 
    \item[\textit{Case 2:}\hspace{-0.5cm}] \hspace{0.5cm} $k_0$ is not a minimum in $K$.
    \\
    We will need the next claim: 
    \begin{afirm}\label{LemaSumasEnQ}
    There exists $\epsilon>0$ such that $k_0+\epsilon\in\Sigma$ and just one of the following holds: 
    \begin{itemize}
        \item[i)] $\displaystyle\sum_{i\in\Sigma,\ k_0<i\leq k_0+\epsilon}\lambda(A_i)\ <\ \sum_{i\in\Sigma,\ k_0<i\leq k_0+\epsilon}\lambda(B_i).$
        \item[ii)] $\displaystyle\sum_{i\in\Sigma,\ k_0<i\leq k_0+\epsilon}\lambda(A_i)\ >\ \sum_{i\in\Sigma,\ k_0<i\leq k_0+\epsilon}\lambda(B_i). $
    \end{itemize}
    \end{afirm}
    \begin{proof}
    Suppose not, then for all $\epsilon>0$ such that $k_0+\epsilon\in\Sigma$ we have that 
    $$\sum_{k_0<i\leq k_0+\epsilon}\lambda(A_i)=\sum_{k_0<i\leq k_0+\epsilon}\lambda(B_i). $$ 
    So fix  $\epsilon>0$ and since $k_0$ is not a minimum of $K$, exist $q\in(k_0,k_0+\epsilon)\cap\Sigma$ such that, without loss of generality, $\lambda(A_q)<\lambda(B_q)$, and by hypothesis since $q=k_0+\delta$ for $\delta>0$ we have that
    $$\sum_{i\in\Sigma,\ k_0<i\leq q}\lambda(A_i)=\sum_{i\in\Sigma,\ k_0<i\leq q}\lambda(B_i) $$
    then
    $$\sum_{i\in\Sigma,\ k_0<i< q}\lambda(A_i)>\sum_{i\in\Sigma,\ k_0<i< q}\lambda(B_i). $$
    But
    $$\hspace{-4.5cm}\sum_{i\in\Sigma,\ k_0<i<q}\lambda(A_i)=\sup_{q'\in(k_0,q)\cap\Sigma}\ \left(\sum_{i\in\Sigma,\ k_0<i\leq q'}\lambda(A_i)\right)>$$
    $$\hspace{7.5cm}\sup_{q'\in(k_0,q)\cap\Sigma}\ \left(\sum_{i\in\Sigma,\ k_0<i\leq q'}\lambda(B_i)\right)=\sum_{i\in\Sigma,\ k_0<i<q}\lambda(B_i)$$
    therefore there exists $q_0'\in(k_0,q)\cap\Sigma$ such that $\displaystyle\sum_{i\in\Sigma,\ k_0<i\leq q_0'}\lambda(A_i)>\sum_{i\in\Sigma,\ k_0<i\leq q_0'}\lambda(B_i)$, which is a contradiction.
    \end{proof}
    Using the previous claim, choose $\epsilon>0$ such that $q=k_0+\epsilon\in[0,1]\cap \Sigma$ and assume that $\displaystyle\sum_{i\in\Sigma,\ k_0<i\leq q}\lambda(A_i)<\sum_{i\in\Sigma,\ k_0<i\leq q}\lambda(B_i)$. Therefore 
    $$\displaystyle\sum_{i\in\Sigma,\ q<i\leq1}\lambda(B_i)<\displaystyle\sum_{i\in\Sigma,\ q<i\leq1}\lambda(A_i).$$
    However, since $\modN_1\ \hookrightarrow\ \modN_2$ and for all $i>q$, $i\in\Sigma$, $\modM_i\not\hookrightarrow\modM_{q}$, we must have that 
    $$\prod_{i\in\Sigma,\ q<i\leq1}\modM_i^{A_i}\hookrightarrow\prod_{i\in\Sigma,\ q<i\leq1}\modM_i^{B_i} $$
    and thus $\displaystyle\sum_{i\in\Sigma,\ q<i\leq1}\lambda(A_i)\leq\displaystyle\sum_{i\in\Sigma,\ q<i\leq1}\lambda(B_i)$, a contradiction.
    \end{itemize}
\end{proof}
We now provide a couple of example of the previous result:
\begin{ej}
 Let $T=ACF_0$, then $T^R$ has the SB-property for separable models.
\end{ej}

\begin{defn}
For every cardinal $\kappa\geq\aleph_0$, let $I(T,\kappa)$ be the number of nonisomorphic models of $T$ of cardinality $\kappa$.
\end{defn}

\begin{defn}
A complete theory $T$ is an \textit{Ehrenfeucht theory} if it is countable, and $1<I(T,\aleph_0)<\aleph_0$.

\end{defn}

\begin{ej}
Let $T$ be an Ehrenfeucht theory with the SB-property for countable models. Then for $T^R$ has the SB-property for separable models.  
\end{ej}
We now show an easy connection between the SB-property of $T$ and the SB-property of $T^R$.
\begin{cor}\label{Cor:SBrand}
Let $T$ be a complete first order theory with at most countably many countable models. Then $T$ has the SB-property for countable models if and only if $T^R$ has the SB-property for separable models.
\end{cor}
\begin{proof}\hspace{0.01cm}
\begin{itemize}
    \item[$\Rightarrow$] It follows from proof of Theorem \ref{GranTeorema}.
    \item[$\Leftarrow$] Let $\modU=\lbrace \modM_i\rbrace_{i<\omega}$ be the set of countable models of $T$ and let  $\modM_{i_1}$, $\modM_{i_2}\in\modU$ be elementarily bi-embeddable models of $T$. Now, consider $\displaystyle\modN_1=\prod_{i<\omega}\modM_i^{A_i}$ and $\displaystyle\modN_2=\prod_{i<\omega}\modM_i^{B_i}$ two separable models of $T^R$ such that $\lambda(A_{i_1})=1$ and $\lambda(B_{i_2})=1.$ Since $\modM_{i_1}\ \hookrightarrow\ \modM_{i_2}$ and $\modM_{i_2}\ \hookrightarrow\ \modM_{i_1}$, we have that $\modN_1\ \hookrightarrow\ \modN_2$ and $\modN_2\ \hookrightarrow\ \modN_1$, then by hypothesis $\modN_1\cong\modN_2$ so by Facts \ref{RandSepNoIsoCaract} and \ref{UniqnessTheoremSep} $\modM_{i_1}\cong\modM_{i_2}$, as desired.
\end{itemize}
\end{proof}

In the next example  we consider an Ehrenfeucht theory $T$ without the SB-property but with the countable SB-property and we show that the associated randomized theory $T^R$ does not have the SB-property. 

\begin{ej}
Let $\modL=\lbrace <,(c_i)_{i<\omega}\rbrace$, where $<$ is a binary relation symbol and for each $i<\omega$, $c_i$ is a constant symbol. Let $T$ be an $\modL$-theory saying that $<$ is a dense linear order without end points and $c_i<c_{i+1}$ for all $i<\omega$. It is well known \cite[Example 2.3.4]{buechler2017essential} that $T$ has $3$ countable models, and it is easy to see that the SB-property holds for countable models. Therefore by Theorem \ref{GranTeorema} the theory $T^R$ has the SB-property for separable models. Note that $T^R$  has $2^{\aleph_0}$ non-isomorphic separable models.

However, $T$ does not have the SB-property. Let $\modM_1=\Q\sqcup\R$ and $\modM_2=\R\sqcup\Q$ be models of $T$ where the constants are interpreted in the copy of $\R$ so that the sequence $(c_i)_{i<\omega}$ converges in $\R$. It is easy to see that $\modM_1$ and $\modM_2$ are elementary bi-embeddable, however they are not isomorphic, since $\modM_1$ has a countable initial segment but $\modM_2$ does not; thus $T$ does not have the SB-property. Let us prove that $T^R$ does not have the SB-property. 

First note that if $\varphi\colon \modM_1\to\modM_2$ is an elementary embedding, then it induces an elementary embedding $\Tilde{\varphi} \colon \modM_1^{[0,1)}\to\modM_2^{[0,1)}$ given by $\Tilde{\varphi}(f)=\varphi\circ f$. So we have that $\modM_1^{[0,1)}$ and $\modM_2^{[0,1)}$ are also elementary bi-embeddable. 

We will now prove that the structures $\modM_1^{[0,1)}$ and $\modM_2^{[0,1)}$ are not isomorphic. Assume, in order to obtain a contradiction, that there exists $\Phi \colon \modM_1^{[0,1)}\to\modM_2^{[0,1)}$ an isomorphism. Let $a\in\modM_1$ be such that $|S_a=\lbrace x\in\modM_1\ |\ x<a\rbrace|=\aleph_0.$ Consider the set  
$$\Omega^a=\left\lbrace f\in\modM_1^{[0,1)}\ |\ \lambda\left(\llbracket f< \chi_a\right\rrbracket)=1\right\rbrace,$$
where $\chi_a\in\modM_1^{[0,1)}$ is such that $\lambda(\llbracket \chi_a=a\rrbracket)=1.$ Note that $\Omega^a$ is separable since $$\Omega_0^a=\lbrace f\in\Omega^a\ |\ |Im(f)|<\aleph_0\rbrace$$
is a separable dense subset, then $\Phi(\Omega^a)$
must be separable and $\Phi(\Omega_0^a)$ a separable dense subset of $\Phi(\Omega^a)$. Also, note that since $\Phi$ is an isomorphism we must have that 
$$\Phi(\Omega^a)=\left\lbrace h\in\modM_2^{[0,1)}\ |\ \lambda(\llbracket h<\Phi(\chi_a)\rrbracket)=1\right\rbrace, $$
but, as we will now show, $\Phi(\Omega^a)$ cannot be separable. Indeed, if we are given $\lbrace h_n\rbrace_{n<\omega}\subseteq \Phi(\Omega^a)$, we have that, for all $n<\omega$, $|Im(h_n)|\leq\aleph_0$, so  $\left|\displaystyle\bigcup_{n<\omega}Im(h_n)\right|=\aleph_0$. 
Since $\Phi(\chi_A)\in \modM_2^{[0,1)}$, we may write $\Phi(\chi_A)=\displaystyle \sum_i  m_i \chi_{B_i}$ where $\{B_i\}_i$ is a measurable partition of $[0,1)$, $m_i\in \modM_2=\mathbb{R} \sqcup \mathbb{Q}$ and we may assume $\mu(B_i)>0$.

Since for all $i$, $\left|S_i=\lbrace x\in\modM_2\ |\ x<m_i \rbrace\right|=2^{\aleph_0}$, we can choose $m_i'\in S_{i}\backslash\displaystyle\bigcup_{n<\omega}Im(h_n)$. Let $h\in\Phi(\Omega^a)$ be defined as 
$$h=\sum_i m_i'\chi_{B_i}.$$ 
Then for all $n<\omega$, $$d(h,h_n)=\lambda(\llbracket h\neq h_n\rrbracket)=1,$$
and thus $\Phi(\Omega^a)$ is not separable, a contradiction. We have shown $\modM_1^{[0,1)}$ and $\modM_2^{[0,1)}$ are two elementarily bi-embeddable models of $T^R$ that are not isomorphic, so $T^R$ does not have the SB-property.

\end{ej}

The arguments in this section depend heavily on the description of separable models of $T^R$ when $I(T,\aleph_0)\leq \aleph_0$. This leaves some open questions:

\begin{preg}
Does $T$ has the SB-property for countable models if and only if $T^R$ has the SB-property for separable models? 
\end{preg}

And more generally

\begin{preg}
Does $T$ has the SB-property if and only if $T^R$ has the SB-property? 
\end{preg}

\section{Superstability and the SB-property}
This section is based on section 5.1 in \cite{goodrick2007does}. In Theorem \ref{TeoSuperstableSB} we prove, in the continuous context, a result analogous to Theorem 5.5 in \cite{goodrick2007does}. This theorem is the first part of the result stated in Fact \ref{superstable-SB} for first order discrete theories.
\\
\\
For the rest of the section we will fix $\modL$ a continuous countable language, and $T$ a stable theory.  
\begin{defn}
Let $\modM\models T$ sufficiently saturated, let $A_0\subseteq A_1 \subset \modM$ be small, $\overline b\in \modM$ and $\epsilon>0$. We say that $tp(\overline b/A_1)$ $\epsilon$-forks over $A_0$ if for all $\overline{b}'\in \modM$ with $d(\overline{b},\overline{b}')<\epsilon$ we have that 
$tp(\overline{b}'/A_1)$ forks over $A_0$. We say that $T$ is \textit{strictly stable} if there is $\epsilon>0$ and an infinite chain of types 
    $$p_0(\overline{x})\subset p_1(\overline{x})\subset p_2(\overline{x})\subset\dots$$
    such that for all $i<\omega$, $p_{i+1}$ $\epsilon$-forks over $dom(p_i).$
\end{defn}
\begin{defn}[\textit{Based on Definition 5.2 in \cite{goodrick2007does}}]\label{Def-aislado-Constru}\hspace{0.01cm} 

    \begin{itemize}
        \item[I.] A type $p\in S(B)$ is $f$-\textit{isolated} if there is a  finite set $A\subseteq B$ such that $p$ does not fork over $A$.
        \item[II.] We say that $B$ is $f$-\textit{constructable over $A$} if $A\subseteq B$,  and there is a enumeration $\lbrace b_i\rbrace_{i<\alpha}$ of a dense subset in $B$ such that for all $i<\alpha$, $tp\left(b_i/A\cup\lbrace b_j\rbrace_{j<i}\right)$ is $f$-isolated.
    \end{itemize}
\end{defn}
\begin{obs}
    In the previous definition, if $b_i\in A$ then $tp(b_i/A\cup\lbrace b_j\rbrace_{j<i})$ is always $f$-isolated, since $b_i\forkindep_{b_i} A\cup\lbrace b_j\rbrace_{j<i}$.
\end{obs}
\begin{obs}
    Recall that a theory $T$ is \textit{superstable} if is stable but not strictly stable.
\end{obs}

\begin{lem}\label{Lem1SuperStab}
For every set $A$ there is a model $\modM\supseteq A$ such that $\modM$ is $f$-constructable over $A$ and $\| \modM\|\leq\|A\|+\|T\|$.
\end{lem}
\begin{proof}
Take a dense subset $A_0\subseteq A$ such that $|A_0|=\|A\|$, and let $\lbrace \varphi_{(i,0)}(x)=0\rbrace_{i\in I}$ be a dense subset of consistent formulas in $T_{A_0}$, the theory obtained after adding constants for $A_0$. Note that, since $T$  is countable, $|I|=\|T\|+|A_0|$. For each $i\in I$, let $b_{(i,0)}\models (\varphi_{(i,0)}(x)=0)$, 
and choose $b_{(i,0)}$ such that $tp(b_{(i,0)}/\overline{a}_{(i,0)}\cup\lbrace b_{(j,0)}\rbrace_{j<i})$ does not fork over $\overline{a}_{(i,0)}$, where $\overline{a}_{(i,0)}$ are the parameters of $\varphi_{(i,0)}(x)$. Let $A_1=A_0 \cup \lbrace b_{(i,0)}\rbrace_{i\in I}$, so by definition $A_1$ is $f$-constructable over $A_0$. Inductively we can construct a chain 
$$A_0\subseteq A_1\subseteq A_2\subseteq A_3\subseteq\dots$$
such that $A_{i+1}$ is $f$-constructable over $A_i$ and call $B=\displaystyle\bigcup_{i<\omega}A_i$. For each $i<\omega$, $|A_i|=\|T\|+|A_0|$, so $|B|=\| T\|+|A_0|$. Now, $\modM=\overline{B}$ is a model of $T$ with $\|\modM\|\leq|B|$ and $A\subseteq\modM$. Let us see that $\modM$ is the model we are looking for. Consider the enumeration of $B$ given by $\lbrace b_{(i,j)}\rbrace_{i\in I,j<\omega}$  with the lexicographic order, where as above $b_{(i,j)}\in A_{j+1}$ and $b_{(i,j)}\models (\varphi_{(i,j)}(x)=0)$, where $\lbrace \varphi_{(i,j)}(x)=0\rbrace_{i\in I}$ is a dense subset of consistent formulas in $T_{A_{j}}$. By construction, for every $(i,j)$ we have that 
$$p_{ij}(x)=tp(b_{(i,j)}/\overline{a}_{(i,j)}\cup\lbrace b_{(i,k)}\rbrace_{k<j}\cup\lbrace b_{(n,\ell)}\rbrace_{n<i,\ell<\omega} )$$ 
does not fork over $\overline{a}_{(i,j)}$, the tuple of parameters for the formula $\varphi_{(i,j)}(x)$, so $p_{ij}(x)$ is $f$-isolated. Thus, by monotonicity we have that $tp(b_{(i,j)}/ A_0\cup\lbrace b_{(i,k)}\rbrace_{k<j}\cup\lbrace b_{(n,\ell)}\rbrace_{n<i,\ell<\omega} )$
is $f$-isolated. Since $A_0$ is dense in $A$, 
$tp(b_{(i,j)}/ A\cup\lbrace b_{(i,k)}\rbrace_{k<j}\cup\lbrace b_{(n,\ell)}\rbrace_{n<i,\ell<\omega} )$
is also $f$-isolated.

\end{proof}

\begin{lem}\label{Lem2SuperStab}
    If $\modM$ if $f$-constructable over $A$, then there exist a dense subset $B\subseteq\modM$ such that for all $\overline{b}\in B$ there is a finite subset $F\subseteq A$ for which  $tp(\overline{b}/A)$ does not fork over $F$.
\end{lem}
\begin{proof}
Since $\modM$ is $f$-constructable take $B_0=\lbrace b_i\rbrace_{i<\alpha}$ dense in $\modM-A$ as in Definition \ref{Def-aislado-Constru} and let $B=B_0\cup A$. It is enough to take $\overline{b}=(b_{i_1},\dots,b_{i_n})\in B-A$ for some $i_1<\dots<i_n<\alpha$ and proof by induction on $i_n$ that there exist $F\subseteq A$ finite such that $tp(\overline{b}/A)$ does not fork over $F$.

The base case is simply the definition of $f$-constructability, so assume it has been proved for all sequences from $B$ that are constructed before the $i_n$-th stage. Because $\modM$ is $f$-constructable, there exists $F_0\subseteq A$ finite and $b_{j_1},\dots,b_{j_m}$ with $j_1<\dots,j_m<i_n$ such that $tp(b_{i_n}/A\cup\lbrace b_k\rbrace_{k<i_n})$ does not fork over $F_0\cup\lbrace b_{j_1},\dots,b_{j_m}\rbrace$. Call $\tilde{b}=(b_{j_1},\dots,b_{j_m},b_{i_1},\dots,b_{i_{n-1}})$, by induction hypothesis there exist $F_1\subseteq A$ finite such that $tp(\tilde{b}/A)$ does not fork over $F_1$ and write $F=F_0\cup F_1$. Since $b_{i_n}\forkindep_{F_0\cup\lbrace b_{j_1},\dots,b_{j_m}\rbrace}A\cup\lbrace b_k\rbrace_{k<i_n}$ by monotonicity we have that $ b_{i_n}\forkindep_{F\cup\tilde{b}}A$. By symmetry 
 \begin{equation}\label{Fork1}
     A\forkindep_{F\cup\tilde{b}}b_{i_n}.
 \end{equation}
  Similarly since $\tilde{b}\forkindep_{F_1}A$ by monotonicity we have that $\tilde{b}\forkindep_F A$ and by symmetry 
  \begin{equation}\label{Fork2}
  A\forkindep_F \tilde{b}.
  \end{equation}
  So using (\ref{Fork1}) and (\ref{Fork2}) by transitivity we have that $A\forkindep_{F} \tilde{b},b_{i_n} $, so by finite character we have that 
  $$A\forkindep_Fb_{i_1},\dots,b_{i_{n-1}},b_{i_n},$$
  and finally by symmetry we get that 
  $$ \overline{b}\forkindep_{F} A$$
 as we wanted.
\end{proof}

\begin{teo}\label{TeoSuperstableSB}
   If $T$ is strictly stable then $T$ does not have the SB-property.
 \end{teo}
\begin{proof}
    Since $T$ is strictly stable there exist an $\epsilon>0$ and a sequence of possibly infinite sets 
    $$A_0\subset A_1\subset A_2\subset\dots$$
    such that $tp(\overline{x}/A_i)\subseteq tp(\overline{x}/A_{i+1})$ and $tp(\overline{x}/A_{i+1})$ $\epsilon$-forks over $A_i$.  Call $A=\displaystyle\bigcup_{i<\omega}A_i$, and let $p(\overline{x})=\displaystyle\bigcup_{i<\omega}tp(\overline{x}/A_i)\in S_{|\overline{x}|}(A)$. Since $T$ is stable, there exist saturated model $\modM$ such that $\|\modM\|$ is at least $\|T\|^+$. Choose $\modM$ such that $\modM\forkindep A$. Now by the Lemma \ref{Lem1SuperStab} there exist a model $\modN$ with $\|\modN\|=\|\modM\|$ and $\modN$  is $f$-constructable over $\modM\cup A$. 
    \begin{afirm}\label{ClaimSuperStab}
        There is no $\overline{b}\in\modN$ such that $\overline{b}\vdash p(\overline{x}).$
    \end{afirm}
    \begin{proof}
        Suppose there is $\overline{b}\in\modN$ such that $\overline{b}\vdash p(\overline{x})$. By Lemma \ref{Lem2SuperStab} there is $\overline{b}'\in\modN$ such that
        $d(b,b')<\epsilon$ and $\overline{b}'\forkindep_F\modM\cup A$, for a finite subset $F\subset \modM\cup A$. So by monotonicity and symmetry
        \begin{equation}\label{Fork3}
            A\forkindep_{\modM\cup A_{i_0}} \overline{b}',
        \end{equation}
  for some $i_0<\omega$. Similarly, since $\modM\forkindep A$ by monotonicity and symmetry we have that 
  \begin{equation}\label{Fork4}
      A\forkindep_{A_{i_0}}\modM,
  \end{equation}
  Thus, using equations (\ref{Fork3}), (\ref{Fork4}), transitivity  and symmetry we get that
  $$\overline{b}'\forkindep_{A_{i_0}} A.$$
  Note that $d(\overline{b},\overline{b}')<\epsilon$ and $\overline{b}'\forkindep_{A_{i_0}} A$,
  a contradiction since  $tp(\overline{b}/A)$ $\epsilon$-forks over $A_{i_0}$.
    \end{proof}
By the Claim \ref{ClaimSuperStab}, the structure $\modN$ cannot be saturated since $|A|<\|\modN\|$, so $\modN$ cannot be isomorphic to $\modM$. But, since $\modN$ is $f$-constructable over over $\modM\cup A$, we have $\modM\ \sbestrucel\ \modN$ and $\modN\ \sbestrucel\ \modM$ because $\modM$ is saturated. So $T$ does not have the SB-property.
\end{proof}

This result allows us to obtain several theories without the SB-property. Let us consider again some theories from Sections \ref{sec_PA} and \ref{sec_Rand}:

\begin{cor}
    The theory $APrA$ does not have the SB-property.
\end{cor}
\begin{proof}
   In \cite{BenYac-Berens-2008perturbations} it is shown that the theory $APrA$ is strictly stable, then by Theorem \ref{TeoSuperstableSB} the theory $APrA$ does not have the SB-property.
\end{proof}

\begin{cor}
    Let $T$ be a strictly stable first order theory and let $T^R$ be the associated randomized theory. Then $T^R$ does not have the SB-property.
\end{cor}
\begin{proof}
   By counting types (for example following the argument of the first part of the proof of Theorem 4.1 in \cite{ben2009randomizations}), it follows that if $T$ is strictly stable, then so is $T^R$. By Theorem \ref{TeoSuperstableSB} the theory $T^R$ does not have the SB-property.
\end{proof}


\addcontentsline{toc}{section}{Referencias}
\bibliographystyle{abbrv}
\bibliography{Bib}
\end{document}